\documentclass{amsart}
\usepackage{amssymb,amsmath, amsfonts,amsrefs, tikz, epsfig, float}
\usepackage{amsthm}
\usepackage{mathrsfs}
\usepackage{enumitem, indentfirst}
\usepackage[fleqn,tbtags]{mathtools}
\usepackage{color}
 \newtheorem{theorem}{Theorem}[section]
 \newtheorem{corollary}[theorem]{Corollary}
 \newtheorem{lemma}[theorem]{Lemma}
 \newtheorem{proposition}[theorem]{Proposition}

 \theoremstyle{definition}
 \newtheorem{example}{Example}[section]
 
 \theoremstyle{remark}
 \newtheorem{remark}[theorem]{Remark}
  \numberwithin{equation}{section}

\newenvironment{enumi}{\begin{enumerate}[label=\textup{(\roman*)}]}{\end{enumerate}}
\newenvironment{enuma}{\begin{enumerate}[label=\textup{(\alph*)}]}{\end{enumerate}}

\renewcommand{\epsilon}{\varepsilon}
\renewcommand{\phi}{\varphi}
\renewcommand{\theta}{\vartheta}
\newcommand{\balg}{\mathscr{B}}
\DeclareMathOperator{\sform}{\mathfrak{s}}
\DeclareMathOperator{\tform}{\mathfrak{t}}

\DeclareMathOperator{\wform}{\mathfrak{w}}

\DeclarePairedDelimiterX\sipt[2]{(}{)_{\tform}}{#1\,\delimsize\vert\,#2}
\DeclarePairedDelimiterX\sipv[2]{(}{)_{v}}{#1\,\delimsize\vert\,#2}
\DeclarePairedDelimiterX\sipw[2]{(}{)_{w}}{#1\,\delimsize\vert\,#2}

\newcommand{\alg}{\mathscr{A}}

\newcommand{\abs}[1]{\lvert#1\rvert}
\newcommand{\dupN}{\mathbb{N}}

\newcommand{\seq}[1]{(#1_{n})_{n\in\dupN}}

\newcommand{\dupR}{\mathbb{R}}
\newcommand{\dupC}{\mathbb{C}}

\newcommand{\ran}{\operatorname{ran}}

\newcommand{\lefpoz}{\mathscr{L}_+(E,F)}

\newcommand{\sigef}{\sigma(E,F)}
\newcommand{\sigfe}{\sigma(F,E)}
\newcommand{\hil}{\mathcal H}

\newcommand{\hila}{\hil_A}
\newcommand{\hilc}{\hil_C}

\DeclarePairedDelimiterX\sip[2]{(}{)}{#1\,\delimsize\vert\,#2}
\DeclarePairedDelimiterX\siptilde[2]{(}{)_{\!_{\widetilde{A}}}}{#1\,\delimsize\vert\,#2}
\DeclarePairedDelimiterX\sipf[2]{(}{)_{f}}{#1\,\delimsize\vert\,#2}
\DeclarePairedDelimiterX\sipg[2]{(}{)_{g}}{#1\,\delimsize\vert\,#2}
\DeclarePairedDelimiterX\siptw[2]{(}{)_{\tform+\wform}}{#1\,\delimsize\vert\,#2}
\DeclarePairedDelimiterX\set[2]{\{}{\}}{#1\,:\,#2}
\DeclarePairedDelimiterX\dual[2]{\langle}{\rangle}{#1,#2}
\DeclarePairedDelimiterX\sipa[2]{(}{)_{\!_A}}{#1\,\delimsize\vert\,#2}
\DeclarePairedDelimiterX\sipc[2]{(}{)_{\!_C}}{#1\,\delimsize\vert\,#2}
\DeclarePairedDelimiterX\sipab[2]{(}{)_{\!_{A+B}}}{#1\,\delimsize\vert\,#2}
\DeclarePairedDelimiterX\sipb[2]{(}{)_{\!_B}}{#1\,\delimsize\vert\,#2}
\newcommand{\anti}[1]{\bar{#1}'}

\allowdisplaybreaks
\title[Means of operators ]{Operators on anti-dual pairs:\\ Means of positive operators}

\author[Zs. Tarcsay]{Zsigmond Tarcsay}

\address{%
Zs. Tarcsay \\ Department of Mathematics\\ Corvinus University of Budapest\\ IX. F\H ov\'am t\'er 13-15.\\ Budapest
H-1093 \\ Hungary\\ and Department of Applied Analysis  and Computational Mathematics\\ E\"otv\"os Lor\'and University\\ P\'azm\'any P\'eter s\'et\'any 1/c.\\ Budapest H-1117\\ Hungary}
\email{zsigmond.tarcsay@uni-corvinus.hu}

\author[\'A. G\"ode]{\'Abel G\"ode}

\address{%
\'A. G\"ode \\ Department of Applied Analysis  and Computational Mathematics\\ E\"otv\"os Lor\'and University\\ P\'azm\'any P\'eter s\'et\'any 1/c.\\ Budapest H-1117\\ Hungary}
\email{godeabel@student.elte.hu}


\begin{document}

\begin{abstract}
We develop a theory of Kubo-Ando type means for positive operators on anti-dual pairs. This framework extends the classical theory of operator means beyond bounded Hilbert space operators and provides a common setting for several concrete classes of positive objects, including nonnegative sesquilinear forms and states on $C^*$-algebras. 
We also introduce a Busch-Gudder type strength function for positive operators on anti-dual pairs and investigate its interaction with operator means.
Applications are given to Hilbert space operators, nonnegative forms, and representable positive functionals.
\end{abstract}

\maketitle

\section{Introduction}

The theory of operator means, developed by Kubo and Ando in their seminal
work \cite{KuboAndo}, provides an axiomatic framework for binary operations on
the cone of bounded positive operators on a Hilbert space. Its basic principles (i.e., monotonicity, the transformer inequality, and continuity from above) bring
under a common perspective several important constructions, most notably the
arithmetic, harmonic, and geometric means. Since its introduction, the theory
has become an indispensable tool in operator theory and matrix analysis, with
close connections to operator-monotone functions, inequalities for positive
operators, nonnegative forms, and mathematical physics.

The purpose of the present paper is to develop a Kubo--Ando-type theory in a
considerably more general framework: that of positive operators on anti-dual
pairs. An anti-dual pair consists of two complex vector spaces linked by a
nondegenerate sesquilinear pairing, but need not carry an underlying norm or
Hilbert space structure. Nevertheless, positive operators between the two
members of the pair retain many of the fundamental features of bounded positive
Hilbert space operators. More importantly, this framework provides a common
operator-theoretic realization of several classes of positive objects that are
usually treated separately. These include bounded positive operators on Hilbert
spaces, nonnegative sesquilinear forms on arbitrary vector spaces,  states on $C^*$-algebras, or more generally, representable positive functionals
on $*$-algebras (without any topology).

A principal tool in this extension is the Hilbert space canonically associated with a positive operator; cf. \cite{TarcsayMathNach}. The resulting factorization makes it possible to transfer suitable Hilbert space constructions to the anti-dual-pair setting without imposing any additional structure on the original spaces. Using this method, we introduce connections and means on the cone $\lefpoz$ of positive operators. Besides the arithmetic mean, we study the parallel sum and the associated harmonic mean, and establish their fundamental variational and order-theoretic properties. We then define the geometric mean through the positive contractions associated with the two operators on the auxiliary
Hilbert space. In particular, the familiar inequalities
\[
A!B\leq A\#B\leq A\nabla B
\]
remain valid in the present generality.

The second main theme of the paper is the Busch--Gudder strength function.
Originally introduced for effects on Hilbert spaces
\cite{busch1999effects}, the strength of a positive operator in a given
direction measures the largest rank-one positive operator in that direction
which is dominated by the operator. We extend this notion to positive operators
on anti-dual pairs and relate it to the order structure. We then investigate its behaviour under the operator means introduced above.

For the parallel sum and the harmonic mean, we prove the exact identities
\[
\lambda_{A:B}=\lambda_A:\lambda_B,
\qquad
\lambda_{A!B}=\lambda_A!\lambda_B,
\]
thereby extending the corresponding results of Ramanantoanina and Titkos
\cite{titkosmanana}. For the arithmetic mean, we establish
\[
\lambda_{A\nabla B}\geq\lambda_A\nabla\lambda_B
\]
and show that equality holds precisely when $A$ and $B$ are linearly
dependent, generalizing a result of Moln\'ar \cite{molnar2018busch}.  The
geometric mean gives rise to a new phenomenon. We prove the inequality
\[
\lambda_{A\#B}\geq\lambda_A\#\lambda_B
\]
and obtain a complete structural characterization of the equality case. In
contrast to the arithmetic mean, equality may occur even when the two
operators are not linearly dependent. As a further consequence of this
characterization, we determine the corresponding absolutely continuous
components in the Lebesgue decompositions of $A$ with respect to $B$ and of
$B$ with respect to $A$.

The final part of the paper illustrates the scope of the general theory. In
the Hilbert space setting, our results specialize to statements about bounded
positive operators. For nonnegative sesquilinear forms, the construction
recovers the parallel sum and the arithmetic, harmonic, and geometric means of
forms \cite{titkos2014means}, together with the corresponding strength functions. For representable
positive functionals on $*$-algebras (respectively, on $C^*$-algebras), the operator means induce natural
means of functionals. We prove that these means remain representable and
describe them in terms of a common GNS representation. The general results on strength functions, including the geometric-mean equality criterion, then admit corresponding formulations for representable functionals. 

The paper is organized as follows. After recalling the necessary background on anti-dual pairs and positive operators, we introduce connections and the principal operator means, and establish the arithmetic--harmonic iteration for the geometric mean. We then extend the Busch--Gudder strength function to the anti-dual-pair setting and investigate its interaction with these means. In particular, we prove exact identities for the parallel sum and the harmonic mean, characterize equality in the inequalities associated with the arithmetic and geometric means, and derive consequences for Lebesgue decomposition. The concluding section is devoted to applications to positive operators on Hilbert spaces, nonnegative sesquilinear forms, and representable positive functionals.

\section{Preliminaries}

Throughout the paper, all vector spaces are assumed to be complex, and inner
products on Hilbert spaces are taken to be linear in the first variable.

Let $E$ and $F$ be complex vector spaces, and suppose that they are connected
by a sesquilinear form
\[
\dual{\cdot}{\cdot}:F\times E\longrightarrow\dupC
\]
which is linear in its first variable, conjugate-linear in its second variable,
and separating in both variables. Thus,
\[
\dual{f}{x}=0\quad(\forall x\in E)
\]
implies $f=0$, 
and similarly,
\[
\dual{f}{x}=0\quad(\forall f\in F)
\]
implies $x=0$. The pair, together with this anti-duality, will be denoted by
$\dual{F}{E}$ and called an \emph{anti-dual pair}.

The anti-duality induces the weak topologies
\[
\sigef
\qquad\text{and}\qquad
\sigfe
\]
on $E$ and $F$, respectively. Unless stated otherwise, both spaces will always
be equipped with these topologies. Thus, convergence in $F$ with respect to
$\sigfe$ is precisely pointwise convergence on $E$.

The most elementary example is obtained from a Hilbert space $\hil$ by taking
$E=F=\hil$ and using its inner product as the anti-duality. If $X$ is an
arbitrary complex vector space and $\bar X^*$ denotes its conjugate algebraic
dual, then $(X,\bar X^*)$, equipped with the evaluation pairing, is also an
anti-dual pair. Likewise, if $X$ is a locally convex Hausdorff space and
$\anti{X}$ denotes its topological anti-dual, then $(X,\anti{X})$ is an
anti-dual pair with respect to
\[
\dual{f}{x}=f(x),
\qquad x\in X,\quad f\in\anti{X}.
\]

We denote by $\mathscr L(E,F)$ the space of all weakly continuous linear
operators from $E$ into $F$, and by $\mathscr L(E)$ the algebra of weakly
continuous linear operators on $E$. If $T\in\mathscr L(E)$, then its adjoint
$T^*:F\to F$ is determined by
\[
\dual{T^*f}{x}=\dual{f}{Tx},
\qquad f\in F,\quad x\in E.
\]
More generally, if $V:E\to\hil$ is weakly continuous, where $\hil$ is a
Hilbert space, then its adjoint $V^*:\hil\to F$ is defined by
\[
\dual{V^*h}{x}=\sip{h}{Vx},
\qquad h\in\hil,\quad x\in E.
\]
Dually, if $J:\hil\to F$ is weakly continuous, then
$J^*:E\to\hil$ is characterized by
\[
\dual{Jh}{x}=\sip{h}{J^*x},
\qquad h\in\hil,\quad x\in E.
\]

A linear operator $A:E\to F$ is called \emph{positive} if
\[
\dual{Ax}{x}\geq0
\qquad (x\in E).
\]
Every positive operator is automatically weakly continuous and Hermitian (see e.g. \cite{TarcsayMathNach}). in
particular,
\[
\dual{Ax}{y}
=
\overline{\dual{Ay}{x}},
\qquad x,y\in E.
\]
Moreover, its associated sesquilinear form satisfies the Cauchy--Schwarz
inequality
\begin{equation}\label{E:CS-positive}
\abs{\dual{Ax}{y}}^2
\leq
\dual{Ax}{x}\dual{Ay}{y},
\qquad x,y\in E.
\end{equation}
The cone of positive operators from $E$ to $F$ will be denoted by $\lefpoz$.
As usual, the order on this cone is defined by
\[
A\leq B
\quad\Longleftrightarrow\quad
\dual{Ax}{x}\leq\dual{Bx}{x}
\qquad(\forall x\in E).
\]
If $A\in\lefpoz$ and $T\in\mathscr L(E)$, then
$T^*AT\in\lefpoz$, since
\[
\dual{T^*ATx}{x}=\dual{ATx}{Tx}\geq0.
\]

For a net $(A_i)$ in $\lefpoz$, we say that $A_i$ converges pointwise to
$A\in\lefpoz$ if
\[
\dual{A_ix}{y}\longrightarrow\dual{Ax}{y}
\qquad (x,y\in E).
\]
For monotone nets, this convergence is determined by the quadratic forms, by
polarization. In particular, the notation
\[
A_i\downarrow A
\]
means that $A_i\geq A_j\geq A$ whenever $i\leq j$ and
\[
\dual{A_ix}{x}\downarrow\dual{Ax}{x}
\qquad(x\in E).
\]
The notation $A_i\uparrow A$ is understood analogously.

We shall impose a mild completeness assumption on the anti-dual pair. The pair
$\dual{F}{E}$ is called \emph{weak-$*$ sequentially complete} if
$(F,\sigfe)$ is sequentially complete. Equivalently, whenever
$(f_n)$ is a sequence in $F$ such that
$(\dual{f_n}{x})_{n\in\dupN}$
is a Cauchy sequence for every $x\in E$, there exists $f\in F$ satisfying
\[
\dual{f_n}{x}\longrightarrow\dual{f}{x}
\qquad(x\in E).
\]
We remark that, according to the Banach--Steinhaus theorem, the anti-dual space of a Banach space (or more generally of a barrelled space) has this property. Furthermore, it is trivial that the anti-dual pair
$(X,\bar X^*)$ is weak-$*$ sequentially complete for every vector space $X$.

We next recall the canonical Hilbert space factorization of a positive
operator. Let $\dual{F}{E}$ be weak-$^*$ sequentially complete and let
$A\in\lefpoz$. On the range space $\ran A$ of $A$ define
\[
\sipa{Ax}{Ay}
\coloneqq
\dual{Ax}{y},
\qquad x,y\in E.
\]
This is a well-defined inner product, and we denote by $\hila$ the completion
of the resulting pre-Hilbert space. By weak-$^*$ sequentially completeness, the canonical inclusion
\[
J_A:\ran A\subseteq\hila\longrightarrow F,
\qquad
J_A(Ax)=Ax,
\]
extends uniquely to a weakly continuous operator
\[
J_A:\hila\longrightarrow F.
\]
Its adjoint satisfies
\begin{equation}\label{E:JAstar}
J_A^*x=Ax,
\qquad x\in E,
\end{equation}
where the vector $Ax\in\ran A$ is regarded as an element of $\hila$.
Consequently,
\begin{equation}\label{E:A=JAJA}
A=J^{}_AJ_A^*,
\qquad
\overline{\ran J_A^*}=\hila.
\end{equation}
Thus every positive operator factors through a Hilbert space. Equivalently,
with $V=J_A^*$, one has $A=V^*V$.

A related factorization will be used repeatedly. Let $A,C\in\lefpoz$ and
assume that $A\leq C$. Then there exists a unique positive contraction
$\widetilde A\in\balg(\hilc)$ such that
\begin{equation}\label{E:A=JCAtJC}
A=J_C\widetilde A J_C^*.
\end{equation}
It is characterized by
\begin{equation}\label{E:Atilde-characterization}
\sipc{\widetilde A J_C^*x}{J_C^*y}
=
\dual{Ax}{y},
\qquad x,y\in E.
\end{equation}
In particular, if $C=A+B$, then the contractions associated with $A$ and
$B$ satisfy
\begin{equation}\label{E:Atilde+Btilde}
\widetilde A+\widetilde B=I_C,
\end{equation}
where $I_C$ denotes the identity operator on $\hilc$.

The Hilbert spaces $\hila$, the canonical embeddings $J_A$, and the
factorizations \eqref{E:A=JAJA}--\eqref{E:Atilde+Btilde} form the basic
technical framework used throughout the paper. Further details concerning
positive operators on anti-dual pairs and these canonical factorizations can
be found in
\cites{TARCSAY2020Lebesgue,TarcsayMathNach, tarcsay-gode2025}.

\section{Means of positive operators}

In what follows, following the classical work of Kubo and Ando \cite{KuboAndo}, we interpret the notions of operator connections and operator means as binary operations on $\mathscr{L}_+(E,F)$.

A binary operation
\[
\sigma:\lefpoz\times\lefpoz\to\lefpoz
\]
is called a connection if it satisfies the following properties for any positive operators $ 
A,B,C,D\in\lefpoz$ 
and any weakly continuous linear operator
$T\in\mathscr L(E)$:
\begin{enumerate}[label=\textup{(\alph*)}]
    \item $A\le B$ and $C\le D$ imply
    $A\sigma C\le B\sigma D,$
    \item
    $T^*(A\sigma B)T\le (T^*AT)\sigma (T^*BT),$
    \item $  A_n\downarrow A
    $ and $B_n\downarrow B$ 
    imply $  A_n\sigma B_n\downarrow A\sigma B.$
\end{enumerate}

A connection $\sigma$ is called a \textit{mean} if it satisfies
\[
A\sigma A=A
\]
for all $A\in\lefpoz$.

The set of connections (respectively, means) is convex, where convex combinations are defined pointwise in the natural way. Furthermore, there is a natural partial order among connections (respectively, means), namely,
\[
\sigma_1\le\sigma_2
\quad\Longleftrightarrow\quad
A\sigma_1B\le A\sigma_2B
\qquad
(\forall A,B\in\lefpoz).
\]

Moreover, the class of means is closed under pointwise limits of decreasing sequences, as shown in the following proposition.

\begin{proposition}
Let $(\sigma_n)_{n\in\mathbb{N}}$ be a monotonically decreasing sequence of connections. Then the limit
\[
A\sigma B\coloneqq \inf_{n\in\mathbb{N}}A\sigma_n B,
\qquad
A,B\in\lefpoz,
\]
defines a connection on $\lefpoz$. If $\sigma_n$ is a mean for all $n\in\mathbb{N}$, then $\sigma$ is itself a mean.
\end{proposition}

\begin{proof}

\textup{(a)}
If $A\leq B$ and $C\leq D$, then
\[
A\sigma C\leq A\sigma_k C\leq B\sigma_k D
\]
for every integer $k$, hence
\[
A\sigma C\le B\sigma D.
\]

\medskip

\textup{(b)}
Let $A,B\in\lefpoz$ and $T\in\mathscr{L}(E)$.
Then
\[
T^*(A\sigma_k B)T
\leq
(T^*AT)\sigma_k (T^*BT)
\]
for every $k$. Thus
\[
\inf_{n\in\mathbb N}
\dual{T^*(A\sigma_n B)Tx}{x}
\leq
\inf_{n\in\mathbb N}
\dual{(T^*AT)\sigma_n (T^*BT)x}{x}
\]
for every $x\in E$. Hence,
\[
T^*(A\sigma B)T
\leq
(T^*AT)\sigma (T^*BT).
\]

\medskip

\textup{(c)}
For monotonically decreasing sequences
\[
A_m\downarrow A,
\qquad
B_m\downarrow B,
\]
we have
\begin{align*}
\dual{(A \sigma B)x}{x}
&=
\inf_{n\in\mathbb N}
\dual{(A\sigma_n B)x}{x} \\
&=
\inf_{n\in\mathbb N}
\left(
\inf_{m\in\mathbb N}
\dual{(A_m\sigma_n B_m)x}{x}
\right) \\
&=
\inf_{m,n\in\mathbb N}
\dual{(A_m\sigma_n B_m)x}{x} \\
&=
\inf_{m\in\mathbb N}
\dual{(A_m\sigma B_m)x}{x}.
\end{align*}

Finally, if $\sigma_n$ is a mean for every $n$, then clearly
\[
A\sigma A
=
\inf_{n\in\mathbb N} A\sigma_n A
=
A.
\]
Thus $\sigma$ is itself a mean.
\end{proof}

The arithmetic mean, together with the left and right trivial means, provides basic examples of operator means on $\lefpoz$. They are defined by
\[
A\nabla B\coloneqq\frac12(A+B),
\qquad
A\omega_l B\coloneqq A,
\qquad
A\omega_r B\coloneqq B.
\]

In what follows, we introduce analogues of the harmonic and geometric means for positive operators, modelled on the Hilbert space theory. In particular, the parallel sum admits several mutually equivalent constructions in this setting as well; see, for example,
\cites{TARCSAY2020Lebesgue, tarcsay2015parallel}.

\subsection{Parallel sum and harmonic mean}

The parallel sum admits several mutually equivalent constructions; see, for instance, 
\cites{TARCSAY2020Lebesgue, tarcsay2015parallel}. For our purposes, the most convenient approach is a representation obtained via factorization through a suitable auxiliary Hilbert space associated with the sum of the operators under consideration.

To this end, let $A,B$ be positive operators on a weak-* sequentially anti-dual pair $\dual{F}{E}$, and let
\[
C\coloneqq A+B
\]
be their sum. Consider the corresponding auxiliary Hilbert space $\hilc$ and the canonical embedding $J_C:\hilc\to F$. Define the densely defined form
\begin{equation*}
    \ran C\times\ran C\to \dupC, 
    \qquad 
    \mathfrak a(Cx,Cy)\coloneqq \dual{Ax}{y},
    \qquad 
    (x,y\in E).
\end{equation*}
It is immediate that $\mathfrak a$ is positive. Moreover,
\begin{equation*}
    \abs{\mathfrak a(Cx,Cy)}^2
    \leq 
    \dual{Ax}{x}\dual{Ay}{y}
    \leq
    \dual{Cx}{x}\dual{Cy}{y}
    =
    \sipc{Cx}{Cx}\sipc{Cy}{Cy},
\end{equation*}
and hence $\mathfrak a$ is continuous. Consequently, there exists a unique positive contraction $\widetilde{A}\in\balg(\hilc)$ such that
\begin{equation*}
    \dual{Ax}{y}
    =
    \sipc{\widetilde A(Cx)}{Cy},
    \qquad x,y\in E.
\end{equation*}
By the same argument, there exists a positive contraction $\widetilde B\in\balg(\hilc)$ such that
\begin{equation*}
    \dual{Bx}{y}
    =
    \sipc{\widetilde B(Cx)}{Cy},
    \qquad x,y\in E.
\end{equation*}
Since $Cx=J_C^*x$, we obtain the useful factorizations
\begin{equation}
    A=J_C\widetilde A J_C^*
    \qquad \mbox{and}\qquad 
    B=J_C\widetilde B J_C^*.
\end{equation}
It is also immediate that
\begin{equation}\label{E:A+B=I}
    \widetilde{A}+\widetilde{B}=I_C,
\end{equation}
where $I_C$ denotes the identity operator on $\hilc$.

With these preparations in place, the parallel sum of $A$ and $B$ can be constructed by the following straightforward procedure.

\begin{theorem}
Let $\dual{E}{F}$ be a weak-* sequentially anti-dual pair and let $A,B\in\lefpoz$ be positive operators. Then the operator
\[
A:B\coloneqq J_C(\widetilde{A}-\widetilde{A}^2)J_C^*
\]
satisfies
\begin{equation}\label{E:parallelquadratic}
    \dual{(A:B)x}{x}
    =
    \inf\set{
    \dual{A(x+y)}{x+y}
    +
    \dual{By}{y}
    }{y\in E}.
\end{equation}
\end{theorem}

\begin{proof}
Fix $x\in E$. First observe that
\begin{equation*}
    \inf_{y\in E}\|\widetilde A J^*_C x+J^*_C y\|^2_C=0,
\end{equation*}
by the density of $\ran J_C^*$ in $\hilc$. Hence,
\begin{align*}
    -\|\widetilde A J^*_C x\|_C^2
    &=
    \inf\set{
    2\operatorname{Re} \sipc{\widetilde A J^*_C x}{J^*_C y}
    +
    \sipc{J_C^*y}{J_C^*y}
    }{y\in E} \\
    &=
    \inf\set{
    2\operatorname{Re} \dual{Ax}{y}
    +
    \dual{(A+B)y}{y}
    }{y\in E}.
\end{align*}
Therefore,
\begin{align*}
    \dual{J_C(\widetilde{A}-\widetilde{A}^2)J_C^*x}{x}
    &=
    \dual{J_C\widetilde{A}J_C^*x}{x}
    -
    \dual{J_C\widetilde{A}^2J_C^*x}{x} \\
    &=
    \dual{Ax}{x}
    -
    \|\widetilde A J^*_C x\|_C^2 \\
    &=
    \dual{Ax}{x}
    +
    \inf\set{
    2\operatorname{Re} \dual{Ax}{y}
    +
    \dual{(A+B)y}{y}
    }{y\in E} \\
    &=
    \inf\set{
    \dual{A(x+y)}{x+y}
    +
    \dual{By}{y}
    }{y\in E},
\end{align*}
which proves the desired identity.
\end{proof}

In view of identity \eqref{E:parallelquadratic}, we call the positive operator
\[
A:B=J_C(\widetilde{A}-\widetilde{A}^2)J_C^*\in\lefpoz
\]
the \textit{parallel sum} of $A$ and $B$; cf. 
\cites{fillmore1971operator,ando1976lebesgue,tarcsay2015parallel,hassi2009lebesgue}.

\begin{remark}
Note that
\[
\widetilde{A}(I_C-\widetilde{A})
=
\widetilde{B}(I_C-\widetilde{B})
\]
by identity \eqref{E:A+B=I}. Consequently,
\[
A:B=B:A.
\]
\end{remark}

\smallskip

The \textit{harmonic mean} of the operators $A$ and $B$ is now defined, in complete analogy with the Hilbert space case, as twice their parallel sum:
\begin{equation}\label{E:harmonicmean}
A ! B \coloneqq  2\, (A : B).
\end{equation}
It is straightforward to verify that the mapping $(A,B) \mapsto A ! B$ defines an operator mean. Indeed, it suffices to check axiom \textup{(b)}, since the remaining axioms are immediate. Let therefore $T\in\mathscr{L}(E)$ be a weakly continuous operator and let $x\in E$. Then
\begin{align*}
    \dual{(T^*AT:T^*BT)x}{x}
    &=
    \inf_{y\in E}
    \left\{
    \dual{T^*AT(x+y)}{x+y}
    +
    \dual{T^*BT y}{y}
    \right\} \\
    &=
    \inf_{y\in E}
    \left\{
    \dual{A\big(Tx+Ty\big)}{Tx+Ty}
    +
    \dual{BTy}{Ty}
    \right\} \\
    &\ge
    \inf_{z\in E}
    \left\{
    \dual{A\big(Tx+z\big)}{Tx+z}
    +
    \dual{Bz}{z}
    \right\} \\
    &=
    \dual{(A:B)Tx}{Tx} \\
    &=
    \dual{T^*(A:B)T x}{x}.
\end{align*}
Thus
\[
T^*(A:B)T
\leq
(T^*AT):(T^*BT),
\]
and consequently
\[
T^*(A!B)T
\leq
(T^*AT)!(T^*BT).
\]
 
In what follows, we will need the following inequality between the arithmetic and harmonic means.

\begin{lemma} \label{L:1.4}
Let $A,B\in\lefpoz$ be positive operators. Then
\begin{equation*}
    A\nabla B\ge A!B.
\end{equation*}
\end{lemma}

\begin{proof}
An immediate calculation gives
\begin{align*}
    A\nabla B-A!B
    &=
    \frac{1}{2}J_C J_C^*
    -
    2J_C(\widetilde{A}-\widetilde{A}^2)J_C^* \\
    &=
    \frac12 J_C
    \Big(I_C-4\widetilde{A}+4\widetilde{A}^2\Big)
    J_C^* \\
    &=
    \frac12 J_C( I_C-2\widetilde{A})^2J_C^*,
\end{align*}
where $I_C$ denotes the identity operator on the Hilbert space $\hilc$.
Since $I_C-2\widetilde{A}$ is self-adjoint, it follows that
\[
J_C( I_C-2\widetilde{A})^2J_C^*
\]
is a positive operator, which proves the lemma.
\end{proof}

\subsection{Geometric mean}

In this subsection, our aim is to define the geometric mean of positive operators.
Recall that if $\hil$ is a Hilbert space and $A,B$ are invertible positive operators on $\hil$, then their geometric mean is defined by
\[
A \# B \coloneqq
A^{1/2}\bigl(A^{-1/2} B A^{-1/2}\bigr)^{1/2} A^{1/2}.
\]
For non-invertible positive operators $A,B$, the geometric mean can be defined, for instance, by setting
\begin{equation}
    A \# B \coloneqq \lim_{\epsilon\to 0} (A+\epsilon I) \# (B+\epsilon I),
\end{equation}
where the limit exists in the operator norm. If $A$ and $B$ commute, then the same holds for their square roots, and in this case one has
\[
A \# B = (AB)^{1/2}.
\]

Of course, in the general anti-dual pair setting, for a positive operator
$A \in\lefpoz$, neither its square nor its square root is defined. Consequently, the geometric mean cannot be introduced directly by imitating the Hilbert space formula. Observe, however, that for the positive operators $\widetilde{A}$ and
$\widetilde{B} = I_C - \widetilde{A}$ acting on the Hilbert space $\mathcal{H}_C$, the parallel sum satisfies
\[
\widetilde{A} : \widetilde{B}
=
\widetilde{A}(I_C - \widetilde{A}).
\]
Consequently, by the definition of the harmonic mean, we have
\[
A ! B
=
J_C ( \widetilde{A} ! \widetilde{B} ) J_C^*.
\]
Similarly, for the arithmetic mean one also has
\[
A \nabla B
=
J_C (\widetilde{A} \nabla \widetilde{B} ) J_C^*.
\]
In light of these observations, it is natural to introduce the \textit{geometric mean} via factorization through the auxiliary Hilbert space $\mathcal{H}_C$ as follows:
\begin{equation}
A \# B
:=
J_C (\widetilde{A}\#\widetilde{B}) J_C^*
=
J_C \sqrt{\widetilde{A}(I_C - \widetilde{A})} J_C^*.
\end{equation}
One readily checks that $\#$ satisfies axioms \textup{(a)--(c)} of connections. Furthermore, if $A=B$, then
\[
\widetilde{A}=\widetilde{B}=\frac12 I_C,
\]
and hence
\[
A\#A
=
\frac12 J_CJ_C^*
=
\frac12 C
=
A.
\]
Thus $\#$ is a mean on $\lefpoz$.

The above definition is intrinsic in the sense that it is expressed entirely in terms of the canonical Hilbert space associated with the sum $C=A+B$. Nevertheless, it is important to verify that it also retains the characteristic approximation properties of the classical geometric mean. In particular, we next show that it arises as the common limit of the arithmetic--harmonic iteration. Since the relevant convergence in the present setting is pointwise convergence with respect to the duality, we first record the following Vigier-type monotone convergence result.

To analyze iterative constructions of means, we require a version of the Vigier theorem \cite{riesz-sz.nagy} adapted to weak-$*$ sequentially complete anti-dual pairs.

\begin{lemma}\label{L:1.5}
Let $\dual{F}{E}$ be a weak-$*$ sequentially complete anti-dual pair and let $\seq A$ be a sequence of positive operators between $E$ and $F$.
\begin{enuma}
    \item If $A_{n+1}\leq A_n$ for all $n$, then $A_n$ converges pointwise to some positive operator $A\in\lefpoz$, i.e.,
    \begin{equation}\label{E:pointwiseconv}
        \dual{A_nx}{y}\to \dual{Ax}{y}
        \qquad (\forall x,y\in E).
    \end{equation}

    \item If $A_{n+1}\geq A_n$ for all $n$ and there is a positive operator $B\in\lefpoz$ such that $A_n\leq B$, then $A_n$ converges pointwise to some positive operator $A\in\lefpoz$ in the sense of \eqref{E:pointwiseconv}.
\end{enuma}
\end{lemma}

\begin{proof}
\textup{(a)}
First observe that for every $x\in E$, the sequence
\[
\alpha_n\coloneqq \dual{A_nx}{x}
\]
is convergent, being non-negative and decreasing. Let now $x,y\in E$ and $n\geq m$. Then $A_m-A_n\geq 0$, so the Cauchy--Schwarz inequality implies
\begin{equation*}
    \abs{\dual{A_mx}{y}-\dual{A_nx}{y}}^2
    \leq
    \dual{(A_m-A_n)x}{x}
    \dual{(A_m-A_n)y}{y},
\end{equation*}
where both factors on the right-hand side tend to $0$ as $n,m\to\infty$. Consequently, $(A_nx)_{n\in\mathbb N}$ is a weak Cauchy sequence in $F$, and therefore it weakly converges to some $Ax\in F$ by weak-$*$ sequential completeness. This means that $A:E\to F$ satisfies \eqref{E:pointwiseconv}. It is clear that $A$ is linear and positive, and hence weakly continuous; that is, $A\in\lefpoz$. This proves statement \textup{(a)}.

\textup{(b)}
This is an immediate application of \textup{(a)} to the decreasing sequence $B-A_n\geq0$.
\end{proof}

We are now in a position to establish the arithmetic--harmonic iteration formula for the geometric mean. This result is the analogue, in the present anti-dual pair setting, of the classical iteration theorem for positive operators; cf. \cite{anderson1979}. It provides an alternative characterization of the above definition and shows that the geometric mean is obtained as the common limit of the decreasing arithmetic and increasing harmonic iterates.

\begin{theorem} \label{iteration}
Let $\dual{F}{E}$ be a weak-* sequentially complete anti-dual pair and let $A,B\in\lefpoz$ be positive operators. Let sequences of operators
\[
(A_n)_{n\in\mathbb{N}},(B_n)_{n\in\mathbb{N}}\subset\lefpoz
\]
be defined by $A_0:=A$, $B_0:=B$ and
\begin{equation}\label{E:AkBk}
    A_{k+1}:=\frac{1}{2}(A_k+B_k),
    \qquad
    B_{k+1}:=A_k!B_k.
\end{equation}
Then
\[
A_n\downarrow A\#B
\qquad\text{and}\qquad
B_n\uparrow A\#B
\]
pointwise on $(E,w(E,F))$.
\end{theorem}

\begin{proof}
As a direct consequence of Lemma~\ref{L:1.4}, we have $A_n \ge B_n$ for all $n \ge 1$.

First, observe that the sequence $(A_n)_{n \in \mathbb{N}}$ is monotonically decreasing, since
\[
    A_n
    =
    \frac{1}{2}(A_n + A_n)
    \ge
    \frac{1}{2}(A_n + B_n)
    =
    A_{n+1}.
\]
Similarly, the sequence $(B_n)_{n \in \mathbb{N}}$ is monotonically increasing, because
\[
    B_n
    =
    2(B_n : B_n)
    \le
    2(A_n : B_n)
    =
    B_{n+1}.
\]

Clearly, $(A_n)$ is bounded below by $0$, and $(B_n)$ is bounded above by $A+B$. Hence, by Lemma~\ref{L:1.5}, both sequences converge pointwise on $(E, w(E,F))$ in the sense of \eqref{E:pointwiseconv}. Next, we show that they have a common limit. Using the monotonicity of $(B_n)$, we compute
\[
    A_{n+1} - B_{n+1}
    =
    \frac{1}{2}(A_n + B_n) - B_{n+1}
    \le
    \frac{1}{2}(A_n + B_n) - B_n
    =
    \frac{1}{2}(A_n - B_n),
\]
which implies
\[
\dual{(A_n - B_n)x}{x} \to 0
\]
for every $x \in E$. Thus $A_n - B_n \to 0$ pointwise in the sense of \eqref{E:pointwiseconv}.

Finally, we show that the common limit is equal to
\[
J_C(\widetilde{A} \# \widetilde{B}) J_C^*.
\]
Define the sequences $(\widetilde{A}_n)$ and $(\widetilde{B}_n)$ in $\balg(\mathcal{H}_C)$ by
\[
\widetilde{A}_0 := \widetilde{A},
\qquad
\widetilde{B}_0 := \widetilde{B} = I_C - \widetilde{A},
\]
and
\[
    \widetilde{A}_{k+1}
    :=
    \widetilde{A}_k \nabla \widetilde{B}_k,
    \qquad
    \widetilde{B}_{k+1}
    :=
    \widetilde{A}_k ! \widetilde{B}_k.
\]
Clearly,
\[
    A_{n+1}
    =
    J_C \widetilde{A}_{n+1} J_C^*
    =
    J_C (\widetilde{A}_n \nabla \widetilde{B}_n) J_C^*,
\]
and
\[
    B_{n+1}
    =
    J_C \widetilde{B}_{n+1} J_C^*
    =
    J_C (\widetilde{A}_n ! \widetilde{B}_n) J_C^*.
\]
By the first part of the proof, $(\widetilde{A}_n)$ and $(\widetilde{B}_n)$ converge pointwise to the same positive operator, say $\widetilde{T} \in \mathcal{B}(\mathcal{H}_C)$. We claim that
\begin{equation}\label{E:Ttilde=AB}
    \widetilde{T}
    =
    \widetilde{A} \# \widetilde{B}
    =
    \widetilde{A}_n \# \widetilde{B}_n
    \quad \text{for every } n.
\end{equation}

Indeed, note that $\widetilde{A}_n$ and $\widetilde{B}_n$ commute for all $n$. Hence
\[
    \widetilde{A}_{n+1} \# \widetilde{B}_{n+1}
    =
    \sqrt{\widetilde{A}_{n+1} \widetilde{B}_{n+1}}
    =
    \sqrt{(\widetilde{A}_n + \widetilde{B}_n)(\widetilde{A}_n : \widetilde{B}_n)}.
\]
Since $\widetilde{A}_1 = \frac{1}{2} I_C$ and $\widetilde{A}_{n+1} \ge \frac{1}{2} \widetilde{A}_n$ for $n \ge 1$, it follows that $\widetilde{A}_n + \widetilde{B}_n$ is invertible for all $n$. Therefore,
\[
    \widetilde{A}_n : \widetilde{B}_n
    =
    \widetilde{A}_n
    (\widetilde{A}_n + \widetilde{B}_n)^{-1}
    \widetilde{B}_n
\]
by the standard formula for the parallel sum in the commuting case; see, for instance, \cite{anderson1979}. Thus
\[
    \widetilde{A}_{n+1} \# \widetilde{B}_{n+1}
    =
    \sqrt{\widetilde{A}_n \widetilde{B}_n}
    =
    \widetilde{A}_n \# \widetilde{B}_n.
\]
Finally, $(\widetilde{A}_n)$ and $(\widetilde{B}_n)$ are uniformly bounded, since
\[
0 \le \widetilde{A}_n, \widetilde{B}_n \le I_C.
\]
Hence
\[
    \widetilde{A} \# \widetilde{B}
    =
    \widetilde{A}_n \# \widetilde{B}_n
    =
    \sqrt{\widetilde{A}_n \widetilde{B}_n}
    \to
    \widetilde{T} \# \widetilde{T}
    =
    \widetilde{T},
\]
which proves \eqref{E:Ttilde=AB}. Therefore,
\[
A \# B = J_C \widetilde{T} J_C^*,
\]
completing the proof.
\end{proof}
\begin{corollary}
    For every pair $A,B$ of positive operators one has 
    \begin{equation*}
        A!B\leq A\#B\leq A\nabla B.
    \end{equation*}
\end{corollary}

\section{Busch--Gudder strength function for positive operators}

The Busch--Gudder strength function of a positive contraction, introduced in \cite{busch1999effects}, is a numerical function associated with a positive operator $A$ and a vector (or state) $h$, which quantifies how strongly the effect ``acts'' in a given direction. It is defined by
\begin{equation}\label{E:lah}
    \lambda(A,h)
    \coloneqq
    \sup\set{\lambda\geq 0}{\lambda P_h \le A},
\end{equation}
where $P_h$ denotes the orthogonal projection onto the one-dimensional subspace spanned by $h$.

The Busch--Gudder strength function can be carried over verbatim to the framework of anti-dual pairs considered in the present paper; cf.~\cite{tarcsaygode2024}. Let $A : E \to F$ be a positive operator. The strength of $A$ along a nonzero state $f \in F$ is defined by
\begin{equation}\label{E:lambdaAf}
\lambda(A,f)
\coloneqq
\sup \{ \lambda \ge 0 : \lambda\, f \otimes f \le A \},
\end{equation}
where $f \otimes f : E \to F$ denotes the rank-one positive operator given by
\[
(f \otimes f)(x) := \overline{\dual{f}{x}}\, f.
\]
Observe that in a Hilbert space, for a unit vector $h$, the operator $h \otimes h$ coincides with the orthogonal projection $P_h$ onto the one-dimensional subspace spanned by $h$. Thus, in this case, the definitions \eqref{E:lah} and \eqref{E:lambdaAf} coincide. Accordingly, in analogy with the Hilbert space case, the scalar-valued mapping $\lambda(A,\cdot)$ will be referred to as the strength function of $A$.

As in the Hilbert space setting, the strength function is closely related to the order structure of positive operators, as well as to the range of the ``square root'' operator $J_A$. According to Theorem~3.1 in \cite{tarcsaygode2024}, for $0\neq f \in F$ we have $\lambda_A(f) > 0$ if and only if $f \in \ran J_A$, and in this case
\begin{equation}\label{E:lambdaAf-formula}
    \lambda_A(f) = \frac{1}{\|J_A^{-1} f\|_A^2}.
\end{equation}
Moreover, the relation $A \leq B$ is equivalent to the condition that
\[
\lambda_A(f) \leq \lambda_B(f)
\]
holds for all $f \in F$. The corresponding Hilbert space version reads as follows; see \cite{busch1999effects} and also \cite{tarcsaygode2024}.

\begin{proposition}
Let $A$ be a positive operator on the Hilbert space $\hil$, and let $h\in\hil$ be a non-zero vector.
\begin{enuma}
    \item $\lambda_A(h)>0$ if and only if $h\in\ran A^{1/2}$.
    \item If $h\in\ran A^{1/2}$, then
    \begin{equation}
        \lambda_A(h)=\frac{1}{\|A^{-1/2}h\|^2}.
    \end{equation}
\end{enuma}
\end{proposition}

In what follows, we establish a useful lemma that extends formula~(3.2) of \cite{titkosmanana}, originally formulated in the Hilbert space setting, to the more general context of anti-dual pairs.

\begin{lemma}
Let $A \in \lefpoz$ and let $f \in F$, $f \neq 0$. Then
\begin{equation}\label{E:TitkosManana}
\lambda(A,f)
=
\inf_{\dual{f}{x} = 1} \dual{Ax}{x}.
\end{equation}
\end{lemma}

\begin{proof}
Let $m_f$ denote the infimum on the right-hand side of \eqref{E:TitkosManana}. We first show that $m_f \le \lambda(A,f)$. Fix $x \in E$. If $\dual{f}{x}=0$, then clearly
\[
0
=
m_f \abs{\dual{f}{x}}^2
=
m_f\dual{(f\otimes f)(x)}{x}
\le
\dual{Ax}{x}.
\]
Otherwise, set
\[
z=\frac{x}{\dual{f}{x}},
\]
so that $\dual{f}{z}=1$. Hence
\[
m_f
\le
\dual{Az}{z}
=
\frac{\dual{Ax}{x}}{\abs{\dual{f}{x}}^2}.
\]
Therefore,
\[
m_f\dual{(f\otimes f)(x)}{x}
\le
\dual{Ax}{x},
\]
and consequently $m_f\leq \lambda(A,f)$.

We now prove the opposite inequality. Let $\lambda > 0$ be such that
\[
\lambda(f \otimes f) \le A,
\]
that is,
\[
\lambda \abs{\dual{f}{x}}^2
\le
\dual{Ax}{x}
\]
for all $x\in E$. Taking any $z \in E$ with $\dual{f}{z}=1$, we obtain
\[
\lambda
=
\lambda \abs{\dual{f}{z}}^2
\le
\dual{Az}{z}.
\]
Taking the infimum over all such $z$ gives $\lambda \le m_f$, and hence $\lambda(A,f)\leq m_f$.
\end{proof}

\begin{proposition}\label{P:mondecreasinglambda}
Let $(A_i)_{i\in \mathscr I}$ be a monotonically decreasing net of positive operators that converges to a positive operator $A \in \lefpoz$ in the sense that
\[
\dual{Ax}{x}
=
\inf_{i\in \mathscr I} \dual{A_i x}{x}
=
\lim_{i\in \mathscr I} \dual{A_i x}{x}.
\]
Then, for every $f \in F$, we have
\[
\lambda_A(f)
=
\inf_{i\in \mathscr I} \lambda_{A_i}(f)
=
\lim_{i\in \mathscr I} \lambda_{A_i}(f).
\]
\end{proposition}

\begin{proof}
Fix $f \in F$. Since for $i \leq j$ we have $A_i \geq A_j \geq A$, it follows that
\[
\lambda_{A_i}(f)
\geq
\lambda_{A_j}(f)
\geq
\lambda_A(f),
\]
and hence
\[
\lambda_A(f)
\leq
\inf_{i\in \mathscr I} \lambda_{A_i}(f)
=
\lim_{i\in \mathscr I} \lambda_{A_i}(f).
\]
Conversely, let
\[
\lambda \leq \inf_{i\in \mathscr I} \lambda_{A_i}(f).
\]
Then for every $x \in E$ and every $i \in \mathscr I$ we have
\[
\lambda \, |\dual{f}{x}|^2
\leq
\dual{A_i x}{x}.
\]
Thus, by the assumption,
\[
\lambda \, |\dual{f}{x}|^2
\leq
\dual{A x}{x},
\]
that is, $\lambda \leq \lambda_A(f)$.
\end{proof}

In what follows, we investigate the relationship between certain operator means and the strength function. More precisely, given an operator mean $\#$ and positive operators $A,B$, we study under what conditions the identity
\[
\lambda(A \# B,\cdot)
=
\lambda(A,\cdot) \# \lambda(B,\cdot)
\]
holds, where in the latter expression the mean $\#$ is understood pointwise.

To this end, the following auxiliary result will be useful.

\begin{lemma}\label{L:2.2}
Let $A,C\in\lefpoz$ be positive operators such that $A\leq C$. Let $\widetilde A\in\balg(\hilc)$ denote the positive contraction in the energy Hilbert space $\hilc$ such that
\[
A=J_C\widetilde{A}J^*_C.
\]
Then
\begin{equation}\label{E:lambdaJCh}
    \lambda_A(J_Ch)=\lambda_{\widetilde{A}}(h)
    \qquad (h\in\hilc).
\end{equation}
\end{lemma}

\begin{proof}
Let $h\in\hilc$ be fixed. Then
\begin{align*}
    \lambda_A(J_Ch)
    &=
    \sup\set{\lambda\ge 0}
    { \lambda\abs{\dual{J_Ch}{x}}^2\le\dual{Ax}{x}\quad (\forall x\in E)}
    \\
    &=
    \sup\set{\lambda\ge 0}
    { \lambda\abs{\sipc{h}{J_C^*x}}^2
    \le
    \sipc{\widetilde A J_C^*x}{J_C^*x}
    \quad (\forall x\in E)}.
\end{align*}
By definition, $\ran J_C^*=\ran C$ is dense in $\hilc$. Therefore, the latter supremum is equal to
\begin{equation*}
   \sup\set{\lambda\ge 0}
   { \lambda\abs{\sipc{h}{k}}^2
   \le
   \sipc{\widetilde Ak}{k}
   \quad (\forall k\in \hilc)}
   =
   \lambda_{\widetilde{A}}(h),
\end{equation*}
which proves \eqref{E:lambdaJCh}.
\end{proof}

We first examine the relationship between the strength function and the harmonic mean. To this end, we need the following lemma, which generalizes a Hilbert space result from \cite{titkosmanana}.

\begin{lemma}
Let $A,B\in\lefpoz$ be positive operators. Then
\begin{equation}
    \lambda_{A:B}(f)
    =
    (\lambda_A:\lambda_B)(f)
    \coloneqq
    \frac{\lambda_A(f)\lambda_B(f)}
         {\lambda_A(f)+\lambda_B(f)},
    \qquad f\in F,
\end{equation}
with the convention that $\frac00\coloneqq 0$ on the right-hand side.
\end{lemma}

\begin{proof}
As before, let $C := A + B$. First observe that
\[
\ran J_A \cup \ran J_B \subseteq \ran J_C.
\]
Therefore, by formula \eqref{E:lambdaAf-formula}, if $f \notin \ran J_C$, then both sides of the formula to be proved are equal to $0$. Hence, we may assume that $f = J_C h$ for some $h \in \hilc$. By Lemma~\ref{L:2.2}, we have
\[
\lambda_A(f) = \lambda_{\widetilde A}(h)
\qquad \text{and} \qquad
\lambda_B(f) = \lambda_{\widetilde B}(h).
\]
By the Key Lemma in \cite{titkosmanana}, one has
\begin{equation*}
    \lambda_{\widetilde{A}:\widetilde{B}}(h)
    =
    \frac{\lambda_{\widetilde{A}}(h)\lambda_{\widetilde{B}}(h)}
         {\lambda_{\widetilde{A}}(h)+\lambda_{\widetilde{B}}(h)}
    =
    \frac{\lambda_A(f)\lambda_B(f)}
         {\lambda_A(f)+\lambda_B(f)}.
\end{equation*}
Using Lemma~\ref{L:2.2} once more, we obtain
\[
\lambda_{A:B}(f)
=
\lambda_{\widetilde{A}:\widetilde{B}}(h),
\]
which completes the proof.
\end{proof}

Since, by definition, the harmonic mean is twice the parallel sum, we immediately obtain the following result.

\begin{theorem}\label{T:lambdaA!B}
For any positive operators $A,B\in\lefpoz$, we have
\[
\lambda_{A!B} = \lambda_A ! \lambda_B.
\]
\end{theorem}


In contrast to the harmonic mean, the arithmetic mean does not interact as well with the strength function; that is, the identity
\[
\lambda_{A\nabla B} = \lambda_A \nabla \lambda_B
\]
holds only in a very special case. Indeed, we have the following result.

\begin{theorem} \label{T:ineqarit}
Let $A,B \in \lefpoz$ be positive operators. Then
\begin{equation}\label{E:strength-nabla}
\lambda_{A\nabla B}(f)
\ge
(\lambda_A \nabla \lambda_B)(f)
\coloneqq
\frac{\lambda_A(f)+\lambda_B(f)}{2}
\end{equation}
for all $f\in F$. The two sides are equal for all $f\in F$ if and only if $A$ and $B$ are linearly dependent.
\end{theorem}

\begin{proof}
Fix a nonzero vector $f\in F$ and consider two nonnegative numbers $\lambda,\mu$ such that
\[
\lambda\leq \lambda_A(f),
\qquad
\mu\leq \lambda_B(f).
\]
Then
\begin{equation*}
    (\lambda+\mu)f\otimes f
    =
    \lambda f\otimes f+\mu f\otimes f
    \leq A+B,
\end{equation*}
and hence
\[
\lambda +\mu\leq \lambda_{A+B} (f).
\]
Consequently,
\[
\lambda_A(f)+\lambda_B(f)\leq \lambda_{A+B} (f),
\]
which yields \eqref{E:strength-nabla}.

Now assume that equality holds in \eqref{E:strength-nabla} for all $f$. It is easy to see that, in this case, the operators $\widetilde A,\widetilde B\in\balg(\hilc)$ also satisfy
\[
\lambda_{\widetilde A}(h) + \lambda_{\widetilde B}(h)
=
\lambda_{\widetilde A + \widetilde B}(h)
\]
for all $h \in \hilc$. According to Proposition~2 in \cite{molnar2018busch}, this is possible only if $\widetilde A$ and $\widetilde B$ are linearly dependent. It then follows from the identities
\[
A = J_C \widetilde A J_C^*,
\qquad
B = J_C \widetilde B J_C^*
\]
that $A$ and $B$ are also linearly dependent.
\end{proof}

After this, we turn to the main result of this section, concerning the relationship between the strength function and the geometric mean. We first show that the inequality
\[
\lambda_{A\# B} \ge \lambda_A \# \lambda_B
\]
always holds.

\begin{proposition}\label{P:2.7}
Let $A,B\in\lefpoz$ be positive operators. Then
\begin{equation}
\lambda_{A\# B} \ge \lambda_A \# \lambda_B.
\end{equation}
\end{proposition}

\begin{proof}
Let us first consider the associated positive operators
\[
\widetilde{A},\widetilde{B}\in\balg(\hilc).
\]
Since they commute, their geometric mean satisfies
\[
\widetilde A \# \widetilde B
=
\widetilde A^{1/2}\widetilde B^{1/2}.
\]
Assume first that both of them are invertible. Then
\begin{align*}
\lambda_{\widetilde{A}\#\widetilde{B}}(h)
&=
\frac{1}{\|(\widetilde{A}\#\widetilde{B})^{-1/2}h\|^2}
=
\frac{1}{\|(\widetilde{A}^{1/2}\widetilde{B}^{1/2})^{-1/2}h\|^2}  \\
&=
\frac{1}{\sip{\widetilde{A}^{-1/2}h}{\widetilde{B}^{-1/2}h}} \\
&\ge
\frac{1}{\|\widetilde{A}^{-1/2}h\|\cdot\|\widetilde{B}^{-1/2}h\|}
=
\sqrt{\lambda_{\widetilde{A}}(h)\lambda_{\widetilde{B}}(h)}  \\
&=
(\lambda_{\widetilde{A}}\#\lambda_{\widetilde{B}})(h).
\end{align*}
If $\widetilde A$ and $\widetilde B$ are arbitrary, not necessarily invertible, positive operators, then
\[
\widetilde A + \frac{1}{n} I \downarrow \widetilde A
\quad \text{and} \quad
\widetilde B + \frac{1}{n} I \downarrow \widetilde B,
\]
and hence
\[
(\widetilde A + \tfrac{1}{n} I)
\#
(\widetilde B + \tfrac{1}{n} I)
\downarrow
\widetilde A \# \widetilde B.
\]
Therefore, by the preceding argument and Proposition~\ref{P:mondecreasinglambda}, we obtain
\[
\lambda_{\widetilde A \# \widetilde B}
=
\inf_{n\in\mathbb{N}}
\lambda_{(\widetilde A+\frac{1}{n}I)\#(\widetilde B+\frac{1}{n}I)}
\ge
\inf_{n\in\mathbb{N}}
\bigl(\lambda_{\widetilde A+\frac{1}{n}I}\#\lambda_{\widetilde B+\frac{1}{n}I}\bigr)
=
\lambda_{\widetilde A}\#\lambda_{\widetilde B}.
\]
Finally, let $f \in F$ be arbitrary. If $f \notin \ran J_C$, then by formula \eqref{E:lambdaAf-formula}, the right-hand side of the desired inequality is $0$. If, on the other hand, $f = J_C h$ for some $h \in \hilc$, then by Lemma~\ref{L:2.2}, taking into account that
\[
A \# B = J_C (\widetilde A \# \widetilde B) J_C^*,
\]
we obtain
\begin{align*}
\lambda_{A \# B}(f)
&=
\lambda_{\widetilde A \# \widetilde B}(h) \\
&\ge
\lambda_{\widetilde A}(h) \# \lambda_{\widetilde B}(h) \\
&=
\lambda_A(f) \# \lambda_B(f).
\end{align*}
Hence $\lambda_{A\#B}\geq \lambda_A\#\lambda_B$, as claimed.
\end{proof}

In the following theorem, we investigate when the strength function of the geometric mean coincides with the geometric mean of the corresponding strength functions. Whereas, for the arithmetic mean, the answer to this question was that equality occurs only when $A$ and $B$ are linearly dependent. As exhibited in the following examples, the situation for the geometric mean is slightly more subtle. 

In our first example we show that the inequality $\lambda_{A\# B}\geq \lambda_A\#\lambda_B$ 
may be strict even for commuting positive definite matrices.
\begin{example}
Let $\hil:=\dupC^3$ and consider the matrices
\[
A=
\begin{bmatrix}
1&0&0\\
0&4&0\\
0&0&1
\end{bmatrix},
\qquad
B=
\begin{bmatrix}
4&0&0\\
0&1&0\\
0&0&1
\end{bmatrix}.
\]
Then $A$ and $B$ commute, and hence
\[
A\#B=(AB)^{1/2}
=
\begin{bmatrix}
2&0&0\\
0&2&0\\
0&0&1
\end{bmatrix}.
\]
For
$h=\frac{1}{\sqrt2}(1,1,0),$ 
Using the formula
\[
\lambda_T(f)=\frac{1}{\langle T^{-1}f,f\rangle}
\]
for positive definite matrices, for the unite vector $f=\frac{1}{\sqrt2}(1,1,0),$ we have
\[
\lambda_A(f)=\lambda_B(f)=\frac85,\qquad \mbox{whereas}\qquad \lambda_{A\#B}(f)=2,
\]
so $\lambda_{A\#B}(f)>(\lambda_A\#\lambda_B)(h).$ 
\end{example}
 
Nevertheless, in contrast to the case of the arithmetic mean, equality may occur even when the operators in question are not linearly dependent, as the following example shows.
\begin{example}
Let $\hil\coloneqq \dupC^3$ and consider the positive operators
\[
    A:=
    \begin{bmatrix}
        \frac14&0&0\\
        0&1&0\\
        0&0&0
    \end{bmatrix},
    \qquad
    B:=
    \begin{bmatrix}
        \frac34&0&0\\
        0&0&0\\
        0&0&1
    \end{bmatrix}.
\]
Clearly, $A$ and $B$ are not linearly dependent. Since they commute, their
geometric mean is given by
\[
    A\#B
    =
    \begin{bmatrix}
        \frac{\sqrt3}{4}&0&0\\
        0&0&0\\
        0&0&0
    \end{bmatrix}.
\]
Thus $\ran (A\#B)^{1/2}=\dupC e_1,$ 
where $e_1=(1,0,0)$ is the first canonical unit vector.
If $f\notin \dupC e_1$, then $f\notin \ran (A\#B)^{1/2}$, and hence
\[
\lambda_{A\#B}(f)=0=\lambda_A(f)\lambda_B(f).
\]
On the other hand, for $f=\alpha e_1$, $\alpha\neq 0$, we have
\[
\lambda_{A\#B}(\alpha e_1)
=
\frac{\sqrt3}{4|\alpha|^2},
\]
while
\[
\lambda_A(\alpha e_1)=\frac{1}{4|\alpha|^2},
\qquad
\lambda_B(\alpha e_1)=\frac{3}{4|\alpha|^2}.
\]
Therefore
\[
\lambda_{A\#B}(\alpha e_1)
=
\sqrt{\lambda_A(\alpha e_1)\lambda_B(\alpha e_1)}
=
(\lambda_A\#\lambda_B)(\alpha e_1).
\]
Consequently, $\lambda_{A\#B}=\lambda_A\#\lambda_B$.
\end{example}

In the next result, we give a precise characterization of the cases in which equality occurs.

\begin{theorem}\label{T:eqgeom}
Let $A,B \in \lefpoz$ be positive operators. Then the following statements are equivalent:
\begin{enumi}
    \item The strength function of the geometric mean of $A$ and $B$ is equal to the geometric mean of their strength functions:
    \begin{equation}\label{E:lambdaAgeomB}
        \lambda_{A \# B} = \lambda_A \# \lambda_B.
    \end{equation}

    \item There exist a Hilbert space $\hil$, a weakly continuous operator
    $V:E\to \hil$ with dense range, pairwise orthogonal closed subspaces
    $\mathcal M,\mathcal N,\mathcal R\subseteq \hil$ such that
    \[
    \mathcal M\oplus \mathcal N\oplus \mathcal R=\hil,
    \]
    and a scalar $0\leq \gamma\leq 1$ such that
    \begin{equation}\label{E:decompAB}
        A
        =
        \gamma V^*P_{\mathcal R}V+V^*P_{\mathcal M}V,
        \qquad
        B
        =
        (1-\gamma) V^*P_{\mathcal R}V+V^*P_{\mathcal N}V.
    \end{equation}
\end{enumi}
\end{theorem}

\begin{proof}
We first prove that \textup{(ii)} implies \textup{(i)}. According to the assumption in \textup{(ii)}, the operators $A$ and $B$ admit the matrix representations
\begin{equation}\label{E:matrixrepr}
    A
    =
    V^*
    \begin{bmatrix}
    I_{\mathcal M}&0&0\\
    0&0&0\\
    0&0&\gamma I_{\mathcal R}
    \end{bmatrix}
    V,
    \qquad
    B
    =
    V^*
    \begin{bmatrix}
    0&0&0\\
    0&I_{\mathcal N}&0\\
    0&0&(1-\gamma) I_{\mathcal R}
    \end{bmatrix}
    V,
\end{equation}
with respect to the orthogonal decomposition
\[
\hil=\mathcal M\oplus \mathcal N\oplus \mathcal R.
\]
In particular,
\[
C=A+B=V^*V.
\]

We next observe that $\hil$ and $\hilc$ are unitarily equivalent. Indeed, the densely defined map
\begin{equation}\label{E:unitary}
 \ran V\to\hilc,
 \qquad
 Vx\mapsto J_C^*x,
\end{equation}
is isometric, since
\begin{align*}
    \|J_C^*x\|_{C}^{2}
    &=
    \dual{Cx}{x}
    =
    \dual{Ax}{x}+\dual{Bx}{x} \\
    &=
    \gamma\dual{V^*P_{\mathcal R}Vx}{x}
    +
    \dual{V^*P_{\mathcal M}Vx}{x}
    +
    (1-\gamma)\dual{V^*P_{\mathcal R}Vx}{x}
    +
    \dual{V^*P_{\mathcal N}Vx}{x} \\
    &=
    \|Vx\|^2
\end{align*}
for every $x\in E$. Since $\ran V$ is dense in $\hil$, the map \eqref{E:unitary} extends to a unitary operator $U:\hil\to\hilc$. Consequently,
\[
    \widetilde{A}
    =
    U(\gamma P_{\mathcal R}+P_{\mathcal M})U^*,
    \qquad
    \widetilde{B}
    =
    U((1-\gamma)P_{\mathcal R}+P_{\mathcal N})U^*.
\]
It follows that
\[
    A\#B
    =
    \sqrt{\gamma(1-\gamma)}\,V^*P_{\mathcal R}V.
\]

We shall use the following simple observation. For every positive operator
$T\in\mathcal B(\hil)$ and every non-zero $h\in\hil$, one has
\begin{equation}\label{E:lambdaVT}
    \lambda_{V^*TV}(V^*h)=\lambda_T(h).
\end{equation}
Indeed, by the density of $\ran V$,
\[
\lambda_{V^*TV}(V^*h)
=
\sup\set{\lambda\geq0}
{\lambda |\langle h,Vx\rangle|^2\leq \langle TVx,Vx\rangle\quad(\forall x\in E)}
=
\lambda_T(h).
\]
Moreover,
\begin{equation}\label{E:lambdaneq0}
    \lambda_{V^*TV}(g)=0
    \qquad
    \bigl(g\in F\setminus \ran V^*T^{1/2}\bigr).
\end{equation}
Indeed, if $\lambda_{V^*TV}(g)>0$, then for some $m>0$,
\[
    |\dual{g}{x}|^2
    \leq
    m\dual{V^*TVx}{x}
    =
    m\|T^{1/2}Vx\|^2
    \qquad (x\in E).
\]
Thus the conjugate-linear functional
\[
    T^{1/2}Vx\mapsto \dual{g}{x}
    \qquad (x\in E)
\]
is well defined and continuous on $\ran T^{1/2}V\subseteq \hil$. By the Riesz representation theorem, there exists $\xi_g\in \hil$ such that
\[
    \dual{g}{x}
    =
    \langle \xi_g,T^{1/2}Vx\rangle
    =
    \dual{V^*T^{1/2}\xi_g}{x}
    \qquad (x\in E).
\]
Consequently, $g=V^*T^{1/2}\xi_g\in\ran V^*T^{1/2}$, proving \eqref{E:lambdaneq0}.

Now let
\[
T:=\gamma P_{\mathcal R}+P_{\mathcal M},
\qquad
S:=(1-\gamma)P_{\mathcal R}+P_{\mathcal N}.
\]
If $g\notin\ran V^*$, then \eqref{E:lambdaneq0} gives
\[
\lambda_A(g)=\lambda_B(g)=\lambda_{A\#B}(g)=0.
\]
On the other hand, if $g=V^*h$ for some $h\in\hil$, then, by \eqref{E:lambdaVT},
\[
\lambda_A(V^*h)=\lambda_T(h),
\qquad
\lambda_B(V^*h)=\lambda_S(h),
\]
and
\[
\lambda_{A\#B}(V^*h)
=
\sqrt{\gamma(1-\gamma)}\,\lambda_{P_{\mathcal R}}(h).
\]
Since $\ran T^{1/2}\cap\ran S^{1/2}=\mathcal R$, we have
\[
\lambda_T(h)\lambda_S(h)=0
\qquad (h\notin \mathcal R).
\]
For $0\neq h\in\mathcal R$, we have
\[
    \lambda_T(h)=\frac{\gamma}{\|h\|^2},
    \qquad
    \lambda_S(h)=\frac{1-\gamma}{\|h\|^2},
    \qquad
    \lambda_{P_{\mathcal R}}(h)=\frac1{\|h\|^2},
\]
hence we obtain
\[
\lambda_{A\#B}(V^*h)
=
\sqrt{\lambda_A(V^*h)\lambda_B(V^*h)}
\]
for every $h\in\hil$. This proves (i).

Next we are going to prove that (i) implies (ii). So assume that equality \eqref{E:lambdaAgeomB} holds. Then the corresponding operators $\widetilde{A}$ and $\widetilde{B}$ on the auxiliary Hilbert space $\hilc$ also satisfy
\begin{equation}\label{eq:main}
\lambda_{\widetilde A \# \widetilde B}
=
\lambda_{\widetilde A} \# \lambda_{\widetilde B}.
\end{equation}
Since the scalar functions associated with the operator means vanish exactly on the intersections of the ranges of the appropriate powers, it follows from \eqref{eq:main} that
\[
\ran(\widetilde A\#\widetilde B)^{1/2}
=
\ran(\widetilde A^{1/2}) \cap \ran(\widetilde B^{1/2})
=
\ran(\widetilde A : \widetilde B)^{1/2}.
\]
Using $\widetilde B = I_C - \widetilde A$, this yields
\[
\ran\bigl(\widetilde A(I_C - \widetilde A)\bigr)^{1/4}
=
\ran\bigl(\widetilde A(I_C - \widetilde A)\bigr)^{1/2}.
\]
This range equality implies that the subspace
\[
\mathcal{R} := \ran\bigl(\widetilde A(I_C - \widetilde A)\bigr)
\]
is closed, according to Theorem~2.6 in \cite{tarcsay2012}; cf. also \cite{fillmore1971operator}. Moreover,
\[
\mathcal{R}
=
\ran(\widetilde A \# \widetilde B)
=
\ran(\widetilde A : \widetilde B).
\]

It is clear that $\mathcal{R}$ is invariant under both $\widetilde A$ and
$\widetilde B = I_C - \widetilde A$. Hence we may consider the restricted operators
\[
\widetilde A_0 := \widetilde A|_{\mathcal{R}},
\qquad
\widetilde B_0 := \widetilde B|_{\mathcal{R}}
\in \balg(\mathcal{R}).
\]
We claim that both $\widetilde A_0$ and $\widetilde B_0$ are boundedly invertible on $\mathcal R$. Indeed, injectivity is immediate. For surjectivity, let $P_{\mathcal R}$ denote the orthogonal projection of $\hilc$ onto $\mathcal{R}$. By Douglas' factorization theorem \cite{douglas}, there exists a bounded operator
\[
D:\hilc\to\ker\bigl(\widetilde A(I_C-\widetilde A)\bigr)^{\perp}
=
\mathcal R
\]
such that
\[
P_{\mathcal R}
=
\widetilde A (I_C - \widetilde A)D.
\]
Thus, for any $z \in \mathcal{R}$, we may write
\[
z
=
P_{\mathcal R}z
=
\widetilde A (I_C - \widetilde A)Dz
=
\widetilde A_0 v
\]
with
\[
v=(I_C-\widetilde A)Dz\in \mathcal{R},
\]
which proves surjectivity. By the Banach isomorphism theorem,
$\widetilde A_0\in\balg(\mathcal R)$ is a topological isomorphism. The same argument shows that $\widetilde B_0$ is invertible as well.

Observe also that
\begin{equation}
    \lambda_{\widetilde{A}}(h)
    =
    \lambda_{\widetilde{A}_0}(h),
    \qquad
    \lambda_{\widetilde{B}}(h)
    =
    \lambda_{\widetilde{B}_0}(h),
    \qquad h\in \mathcal{R}, \ h\neq 0.
\end{equation}
Hence, for such $h$,
\begin{align*}
    \lambda_{\widetilde A_0 \# \widetilde B_0}(h)
    &=
    \frac{1}{\langle(\widetilde A_0\widetilde B_0)^{-1/2}h,h\rangle}
    =
    \frac{1}{\langle \widetilde B_0^{-1/2}h,\widetilde A_0^{-1/2}h\rangle} \\
    &\geq
    \frac{1}{\|\widetilde A_0^{-1/2}h\|\,
    \|\widetilde B_0^{-1/2}h\|}
    =
    (\lambda_{\widetilde A_0} \# \lambda_{\widetilde B_0})(h).
\end{align*}
By \eqref{eq:main}, equality holds throughout. Hence the equality condition in the Cauchy--Schwarz inequality implies that, for every $h\in\mathcal R$, there exists a scalar $\alpha(h)\in\mathbb C$ such that
\[
\widetilde A_0 h = \alpha(h)\widetilde B_0 h.
\]
A standard argument shows that $\alpha(h)$ is independent of $h$, that is,
$\alpha(h)\equiv \alpha$ for some $\alpha\geq0$, and hence
\[
\widetilde A_0=\alpha \widetilde{B}_0.
\]
Since $\widetilde B_0=I_{\mathcal R}-\widetilde A_0$, we obtain
\[
\widetilde A_0 = \gamma I_{\mathcal{R}},
\qquad
\widetilde B_0 = (1-\gamma) I_{\mathcal{R}}
\]
for some $0\leq \gamma\leq 1$.

This yields a decomposition of $\hilc$ as follows. On $\mathcal R^\perp$, the operators
\[
\widetilde{A}_1:=\widetilde A|_{\mathcal R^\perp},
\qquad
\widetilde{B}_1:=\widetilde B|_{\mathcal R^\perp}
\]
are positive operators satisfying
\[
\widetilde{A}_1\widetilde{B}_1=0,
\qquad
\widetilde{B}_1=I_{\mathcal R^\perp}-\widetilde{A}_1.
\]
Thus $\widetilde{A}_1^2=\widetilde{A}_1$, and therefore $\widetilde A_1$ and $\widetilde B_1$ are complementary orthogonal projections. Let
\[
\mathcal M:=\ran \widetilde A_1,
\qquad
\mathcal N:=\ran \widetilde B_1.
\]
Then
\[
\hilc=\mathcal M\oplus \mathcal N\oplus \mathcal R.
\]
Finally, taking
\[
\hil:=\hilc,
\qquad
V:=J_C^*,
\]
we have $\ran V=\ran J_C^*$ dense in $\hilc$, and
\[
A
=
\gamma V^*P_{\mathcal R}V+V^*P_{\mathcal M}V,
\qquad
B
=
(1-\gamma)V^*P_{\mathcal R}V+V^*P_{\mathcal N}V.
\]
This proves \textup{(ii)}.
\end{proof}
 
Before stating the final result of this section, we briefly recall the general Lebesgue decomposition theorem for positive operators on anti-dual pairs. A detailed account of this topic can be found in \cite{TARCSAY2020Lebesgue}; see also \cite{tarcsay-gode2025}.

For any two positive operators $A,B\in\lefpoz$, the pointwise limit
\begin{equation}
    [B]A\coloneqq \lim_{n\to\infty} A:(nB)
\end{equation}
defines a positive operator in $\lefpoz$. Moreover, the decomposition
\[
A=[B]A+(A-[B]A)
\]
is a Lebesgue-type decomposition of $A$ with respect to $B$, in the sense that
\begin{equation*}
   [B]A \ll B,
   \qquad
   A-[B]A\perp B.
\end{equation*}
Furthermore, $[B]A$ is maximal in the following sense: whenever $C\leq A$ and $C\ll B$, it follows that $C\leq [B]A$. The operators $[B]A$ and $A-[B]A$ are called the absolutely continuous and singular parts of $A$ with respect to $B$, respectively. For further details, we refer the reader to \cite{TARCSAY2020Lebesgue}.

In the following corollary, we determine the absolutely continuous components $[B]A$ and $[A]B$ under the equivalent conditions of Theorem~\ref{T:eqgeom}.

\begin{corollary}\label{C:Lebesgue}
Suppose that the positive operators $A,B\in\lefpoz$ satisfy the equivalent conditions of Theorem~\ref{T:eqgeom}, and let
\[
A
=
\gamma V^*P_{\mathcal R}V+V^*P_{\mathcal M}V,
\qquad
B
=
(1-\gamma)V^*P_{\mathcal R}V+V^*P_{\mathcal N}V
\]
be a representation as in Theorem~\ref{T:eqgeom}. Then
\begin{equation*}
    [B]A
    =
    \eta_\gamma V^*P_{\mathcal R}V,
    \qquad
    [A]B
    =
    \theta_\gamma V^*P_{\mathcal R}V,
\end{equation*}
where
\[
\eta_\gamma
=
\begin{cases}
\gamma, & 0\leq \gamma<1,\\
0, & \gamma=1,
\end{cases}
\qquad
\theta_\gamma
=
\begin{cases}
1-\gamma, & 0<\gamma\leq 1,\\
0, & \gamma=0.
\end{cases}
\]
In particular, if $0<\gamma<1$, then
\begin{equation*}
    [B]A=\gamma V^*P_{\mathcal R}V,
    \qquad
    [A]B=(1-\gamma)V^*P_{\mathcal R}V.
\end{equation*}
\end{corollary}

\begin{proof}
Put
\[
T:=\gamma P_{\mathcal R}+P_{\mathcal M},
\qquad
S:=(1-\gamma)P_{\mathcal R}+P_{\mathcal N}.
\]
Then
\[
A=V^*TV,
\qquad
B=V^*SV.
\]
Since $\ran V$ is dense, the same argument as in the proof of Theorem~\ref{T:eqgeom} shows that the parallel sums are transported through $V$:
\[
A:(nB)=V^*\bigl(T:(nS)\bigr)V,
\qquad n\in\mathbb N.
\]
Consequently,
\[
[B]A
=
V^*([S]T)V.
\]

It remains to compute $[S]T$ on the orthogonal decomposition
\[
\hil=\mathcal M\oplus\mathcal N\oplus\mathcal R.
\]
On $\mathcal M$, we have $T=I_{\mathcal M}$ and $S=0$, hence the absolutely continuous part of $T$ with respect to $S$ is zero. On $\mathcal N$, we have $T=0$, and hence again there is no contribution. On $\mathcal R$, the restrictions are scalar multiples of the identity:
\[
T|_{\mathcal R}=\gamma I_{\mathcal R},
\qquad
S|_{\mathcal R}=(1-\gamma)I_{\mathcal R}.
\]
Therefore
\[
T:(nS)
=
\frac{n\gamma(1-\gamma)}
     {\gamma+n(1-\gamma)}
P_{\mathcal R}.
\]
Letting $n\to\infty$, we obtain
\[
[S]T
=
\eta_\gamma P_{\mathcal R}.
\]
Thus $[B]A=\eta_\gamma V^*P_{\mathcal R}V.$

The formula for $[A]B$ follows by the same argument, interchanging the roles of $A$ and $B$. Indeed,
\[
[A]B
=
V^*([T]S)V,
\]
and on $\mathcal R$ one has
\[
S:(nT)
=
\frac{n\gamma(1-\gamma)}
     {1-\gamma+n\gamma}
P_{\mathcal R}.
\]
Taking the limit gives
\[
[T]S=\theta_\gamma P_{\mathcal R},
\]
and hence $[A]B=\theta_\gamma V^*P_{\mathcal R}V.$
\end{proof}


\section{Applications}

In this section, we discuss several applications of the general theory developed in the previous sections. Our aim is to illustrate how the results obtained for positive operators on anti-dual pairs specialize to concrete settings of independent interest.

\subsection{Operators on Hilbert spaces}

We begin with the classical Hilbert space setting, which served as the principal motivation for the present generalization. Thus, throughout this subsection, we assume that $E=F=\hil$, where $\hil$ is a Hilbert space, and that the anti-duality coincides with the inner product of $\hil$.

Naturally, all of the results established above remain valid in this particular case. With the exception of Theorem~\ref{T:eqgeom} and Corollary~\ref{C:Lebesgue}, however, they do not substantially extend the corresponding classical results already available in the literature. On the other hand, these two statements admit particularly transparent reformulations in certain special situations, which are worth recording separately.

First, we collect the direct consequences concerning the strength function.

\begin{corollary}
Let $\hil$ be a Hilbert space and let $A,B\in\mathcal B(\hil)_+$ be positive operators. Then the following assertions hold:
\begin{enuma}
    \item $ \lambda_{A:B}=\lambda_A:\lambda_B$, and
        $\lambda_{A!B}=\lambda_A!\lambda_B.$
        \item
        $\lambda_{A\nabla B}\geq \lambda_A\nabla\lambda_B$.  Equality holds for all vectors if and only if $A$ and $B$ are linearly dependent.
    \item 
        $\lambda_{A\#B}\geq \lambda_A\#\lambda_B.$ Equality holds for all vectors if and only if there exist a Hilbert space $\mathcal K$, a bounded operator
    $V:\hil\to \mathcal K$ with dense range, pairwise orthogonal closed subspaces
    $\mathcal M,\mathcal N,\mathcal R\subseteq \mathcal K$ such that
    \[
        \mathcal K=\mathcal M\oplus\mathcal N\oplus\mathcal R,
    \]
    and a scalar $0\leq \gamma\leq 1$ such that
    \[
        A=\gamma V^*P_{\mathcal R}V+V^*P_{\mathcal M}V,
        \qquad
        B=(1-\gamma)V^*P_{\mathcal R}V+V^*P_{\mathcal N}V.
    \]
\end{enuma}
\end{corollary}


\subsection{Nonnegative forms}
In this subsection, we apply our results to the case of nonnegative sesquilinear forms on a complex vector space, and show how the results of \cite{titkos2014means} can be recovered directly from the theory developed above.

Let $X$ be a complex vector space, and let $\sform$ and $\tform$ be nonnegative sesquilinear forms on $X$. Denote by $\bar{X}^{*}$ the conjugate algebraic dual of $X$, that is, the vector space of all conjugate-linear functionals on $X$. Then $(X,\bar{X}^{*})$ forms a natural anti-dual pair with respect to the canonical duality
\[
\dual{f}{x}:=f(x),
\qquad x\in X,\ f\in\bar{X}^{*}.
\]
Clearly, $(X,\bar{X}^{*})$ is weak-$*$ sequentially complete, indeed complete.

Moreover, the forms $\sform$ and $\tform$ naturally induce positive operators
\[
S,T:X\to \bar{X}^{*}
\]
defined by
\[
\dual{Sx}{y}\coloneqq \sform(x,y),
\qquad
\dual{Tx}{y}\coloneqq \tform(x,y),
\qquad x,y\in X.
\]
Consider the operator $C:=S+T$, the associated Hilbert space $\hil_C$, the canonical embedding
\[
J_C:\hil_C\to \bar X^{*},
\]
and the corresponding positive operators $\widetilde S$ and $\widetilde T$ on $\hil_C$. Then the forms $\sform$ and $\tform$ admit the following representations:
\[
\dual{J_C^{}\widetilde S J_C^*x}{y}
=
\sform(x,y),
\qquad
\dual{J_C^{}\widetilde T J_C^*x}{y}
=
\tform(x,y),
\qquad x,y\in X.
\]
(Observe that the associated Hilbert space $\hil_C$ is unitarily equivalent to the Hilbert space $\hil_{\sform+\tform}$ defined in the usual way: namely, one equips the quotient space $X/\ker (\sform+\tform)$ with the inner product
\[
\sip{x+\ker (\sform+\tform)}{ y+\ker (\sform+\tform)}
=
\sform(x,y)+\tform(x,y),
\]
and then takes its completion, see \cite{pusz1975functional}.) 

Consequently, the parallel sum, as well as the arithmetic, geometric, and harmonic means of the forms, can be introduced through the associated positive operators as follows:
\begin{align*}
    (\sform:\tform)(x,y)&\coloneqq\dual{(S:T)x}{y},\\
    (\sform\nabla\tform)(x,y)&\coloneqq\dual{(S\nabla T)x}{y}=\frac12(\sform+\tform)(x,y),\\
    (\sform\#\tform)(x,y)&\coloneqq\dual{(S\#T)x}{y},\\
    (\sform!\tform)(x,y)&\coloneqq\dual{(S!T)x}{y}=2(\sform:\tform)(x,y).\\
\end{align*}
Observe that the quadratic form associated with the form $\sform:\tform$ satisfies
\begin{equation*}
    (\sform:\tform)(x,x)
    =
    \inf_{y\in X}
    \{\sform(y,y)+\tform(x-y,x-y)\}.
\end{equation*}
This formula was used as the definition of the parallel sum in \cite{titkos2014means}.
An immediate consequence of Theorem \ref{iteration} is the following result (cf. \cite{titkos2014means}, Theorem~3.4):
\begin{corollary}
Define the sequences by 
 $\sform_0:=\sform$, $\tform_0:=\tform$ and
\begin{equation}\label{E:sk=tk}
    \sform_{k+1}:=\frac{1}{2}(\sform_k+\tform_k),
    \qquad
    \tform_{k+1}:=\sform_k!\tform_k.
\end{equation}
Then
\[
\sform_n\downarrow \sform\#\tform
\qquad\text{and}\qquad
\tform_n\uparrow \sform \#\tform
\]
pointwise on $X$. In particular, 
\[\sform!\tform\leq \sform\#\tform\leq \sform\nabla\tform.\] 
\end{corollary}
The strength of the form $\sform$ along a nonzero functional $f\in X^{*}$ can be defined by
\begin{equation}
\lambda_{\sform}(f)
\coloneqq
\sup\set{\lambda\geq 0}
{\lambda\abs{f(x)}^2\leq \sform(x,x)\quad (\forall x\in X)}.
\end{equation}
This coincides with the strength of the associated operator $S$ in the direction of $f$. We call the function $\lambda_{\sform}:X^*\setminus\{0\}\to\dupR_+$ the \textit{strength function} of the form $\sform$.

In the following result, we summarize the consequences of the results obtained for the strength function in the setting of nonnegative forms.
\begin{corollary}
Let $\sform$ and $\tform$ be nonnegative forms on the complex vector space $X$. Then the following assertions hold:
\begin{enuma}
    \item $\sform\leq \tform$ if and only if $\lambda_{\sform}\leq \lambda_{\tform}.$
    \item $\lambda_{\sform:\tform}=\lambda_{\sform}:\lambda_{\tform},$ and $\lambda_{\sform!\tform}=\lambda_{\sform}!\lambda_{\tform}.$
    \item $\lambda_{\sform\nabla\tform}\geq
        \lambda_{\sform}\nabla\lambda_{\tform}.$
    Equality holds if and only if $\sform$ and $\tform$ are linearly dependent.

    \item $\lambda_{\sform\#\tform}\geq
        \lambda_{\sform}\#\lambda_{\tform}.$ 
     Equality holds if and only if there exist a Hilbert space $\hil$, a linear operator $V:X\to\hil$ with dense range, pairwise orthogonal closed subspaces
    $\mathcal M,\mathcal N,\mathcal R\subseteq \hil$ such that
    \[
        \hil=\mathcal M\oplus \mathcal N\oplus \mathcal R,
    \]
    and a scalar $0\leq \gamma\leq 1$ such that, for all $x,y\in X$,
    \begin{align*}
        \sform(x,y)
        &=
        \gamma\sip{P_{\mathcal R}Vx}{P_{\mathcal R}Vy}
        +
        \sip{P_{\mathcal M}Vx}{P_{\mathcal M}Vy},\\
        \tform(x,y)
        &=
        (1-\gamma)\sip{P_{\mathcal R}Vx}{P_{\mathcal R}Vy}
        +
        \sip{P_{\mathcal N}Vx}{P_{\mathcal N}Vy}.
    \end{align*}
\end{enuma}
\end{corollary} 
\subsection{Representable functionals}

Let $\alg$ be a unital $*$-algebra, and let $w$ and $v$ be two representable positive functionals on $\alg$, that is, assume that
\[
w(x^*y^*yx)\leq M_y w(x^*x),
\qquad
v(x^*y^*yx)\leq M'_y v(x^*x)
\]
for all $x,y\in\alg$ and for some constants $M_y,M'_y\geq 0$ depending on $y$.

Consider the anti-dual pair $(\alg,\bar{\alg}^*)$, where $\bar{\alg}^*$ denotes the conjugate algebraic dual of $\alg$, endowed with the natural anti-duality. Define positive operators $A,B:\alg\to\bar{\alg}^*$ by
\[
\dual{Ax}{y}\coloneqq w(y^*x),
\qquad
\dual{Bx}{y}\coloneqq v(y^*x),
\qquad x,y\in\alg.
\]
If $\alg$ is a $C^*$-algebra, then every positive functional on $\alg$ is automatically continuous and representable. In that case, one may instead work with the anti-dual pair $(\alg,\anti{\alg})$, where $\anti{\alg}$ denotes the topological anti-dual of $\alg$.

As before, let
\[
C:=A+B,
\]
and consider the associated Hilbert space $\hilc$ together with the canonical linear map
\[
J_C:\hilc\to\bar{\alg}^*.
\]
The positive functional
\[
\rho\coloneqq w+v
\]
is also representable. Indeed, let $\pi_\rho(a)$ be the operator initially defined on $\ran C$ by
\[
\pi_\rho(a)Cx\coloneqq C(ax),
\qquad x\in\alg.
\]
Then, by the assumptions above,
\[
\|\pi_\rho(a)Cx\|_C^2
=
\rho(x^*a^*ax)
\leq
(M_a+M'_a)\rho(x^*x)
=
(M_a+M'_a)\|Cx\|_C^2.
\]
Hence $\pi_\rho(a)$ extends uniquely to a bounded operator on $\hilc$, still denoted by $\pi_\rho(a)$, with
\[
\|\pi_\rho(a)\|\leq \sqrt{M_a+M'_a}.
\]
It is straightforward to verify that
\[
\pi_\rho:\alg\to\balg(\hilc)
\]
is a $*$-representation and that
\begin{equation*}
    \rho(a)
    =
    \sipc{\pi_\rho(a)C1}{C1}
    =
    \sipc{\pi_\rho(a)J_C^*1}{J_C^*1},
    \qquad a\in\alg.
\end{equation*}

Now consider the positive contraction $\widetilde A\in\balg(\hilc)$ associated with $A$, that is, the operator satisfying
\[
J_C\widetilde A J_C^* = A.
\]
Then
\[
w(a)
=
\sipc{\widetilde A J_C^*a}{J_C^*1},
\]
and similarly, with $\widetilde B=I_C-\widetilde A$,
\[
v(a)
=
\sipc{\widetilde B J_C^*a}{J_C^*1}
=
\sipc{(I_C-\widetilde A)J_C^*a}{J_C^*1}.
\]

We claim that $\pi_\rho$ commutes with $\widetilde{A},$ that is,
\begin{equation}\label{E:ropi=piro}
    \pi_\rho(a)\widetilde{A}
    =
    \widetilde{A}\pi_\rho(a),
    \qquad a\in\alg.
\end{equation}
Indeed, for every $x,y\in\alg$ we have
\begin{align*}
     \sipc{\pi_\rho(a)\widetilde{A}J_C^*x}{J_C^* y}
     &=
     \sipc{\widetilde{A}J_C^*x}{J_C^*a^*y} \\
     &=
     w(y^*ax) \\
     &=
     \sipc{\widetilde{A}J_C^*(ax)}{J_C^*y} \\
     &=
     \sipc{\widetilde{A}\pi_\rho(a)J_C^*x}{J_C^*y}.
\end{align*}
Since $\ran J_C^*$ is dense in $\hilc$, this proves \eqref{E:ropi=piro}.

As a consequence, $\pi_\rho(a)$ commutes with $\widetilde{A}^{1/2}$ for every $a\in\alg$, and we obtain the following representation of $w$:
\begin{align*}
    w(a)
    &=
    \sipc{\widetilde{A}\pi_\rho(a)J_C^*1}{J_C^*1} \\
    &=
    \sipc{\pi_\rho(a)\widetilde{A}^{1/2}J_C^*1}
    {\widetilde{A}^{1/2}J_C^*1}.
\end{align*}
Thus, with
\[
\zeta_w\coloneqq \widetilde{A}^{1/2}J_C^*1,
\]
we have
\begin{equation}\label{E:w(a)=}
    w(a)=\sipc{\pi_\rho(a)\zeta_w}{\zeta_w}
    \qquad (a\in\alg).
\end{equation}
Analogously, with
\[
\zeta_v\coloneqq (I_C-\widetilde{A})^{1/2}J_C^*1,
\]
we have
\begin{equation}\label{E:v(a)=}
    v(a)=\sipc{\pi_\rho(a)\zeta_v}{\zeta_v}
    \qquad (a\in\alg).
\end{equation}

In view of the above, we define the parallel sum, as well as the harmonic, arithmetic, and geometric means of the functionals $w$ and $v$, through the corresponding positive operators as follows:
\begin{alignat*}{2}
(w:v)(a)      &:= \dual{(A:B)a}{1},      \qquad &
(w!v)(a)      &:= \dual{(A!B)a}{1},\\
(w\nabla v)(a)&:= \dual{(A\nabla B)a}{1}, \qquad &
(w\#v)(a)     &:= \dual{(A\#B)a}{1}.
\end{alignat*}

\begin{theorem}
The functionals $w:v$, $w!v$, $w\nabla v$, and $w\#v$ are representable positive functionals on the $*$-algebra $\alg$. More precisely:
\begin{enuma}
    \item With $\zeta_{w:v}
    :=
    \sqrt{\widetilde{A}:\widetilde{B}}\,J_C^*1,$
        we have
    \[
    (w:v)(a)
    =
    \sipc{\pi_\rho(a)\zeta_{w:v}}{\zeta_{w:v}}.
    \]

    \item With $\zeta_{w!v}
    :=
    \sqrt{\widetilde{A}!\widetilde{B}}\,J_C^*1,$
      we have
    \[
    (w!v)(a)
    =
    \sipc{\pi_\rho(a)\zeta_{w!v}}{\zeta_{w!v}}.
    \]

    \item With $\zeta_{w\nabla v}
    :=
    \sqrt{\widetilde{A}\nabla\widetilde{B}}\,J_C^*1,$
    we have
    \[
    (w\nabla v)(a)
    =
    \sipc{\pi_\rho(a)\zeta_{w\nabla v}}{\zeta_{w\nabla v}}.
    \]
    \item With $\zeta_{w\# v}
    :=
    \sqrt{\widetilde{A}\#\widetilde{B}}\,J_C^*1,$  we have
    \[
    (w\#v)(a)
    =
    \sipc{\pi_\rho(a)\zeta_{w\#v}}{\zeta_{w\#v}}.
    \]
\end{enuma}
\end{theorem}

\begin{proof}
Since $\widetilde B=I_C-\widetilde A$, the operators
\[
\widetilde{A}:\widetilde{B},
\qquad
\widetilde{A}!\widetilde{B},
\qquad
\widetilde{A}\nabla\widetilde{B},
\qquad
\widetilde{A}\#\widetilde{B}
\]
are all obtained from $\widetilde A$ by continuous functional calculus. Hence, by \eqref{E:ropi=piro}, they all commute with $\pi_\rho(a)$ for every $a\in\alg$.

For the parallel sum, we get
\begin{align*}
(w:v)(a)
=
\dual{(A:B)a}{1}
&=
\sipc{(\widetilde{A}:\widetilde{B})\pi_{\rho}(a)J_C^*1}{J_C^*1} \\
&=
\sipc[\big]{\pi_{\rho}(a)\sqrt{\widetilde{A}:\widetilde{B}}\,J_C^*1}
{\sqrt{\widetilde{A}:\widetilde{B}}\,J_C^*1}.
\end{align*}
The remaining three cases are obtained in an analogous way. 
\end{proof}

The following result is an immediate consequence of Theorem~\ref{iteration}.

\begin{corollary}
Define two sequences of functionals $(w_k)_{k\in\mathbb N}$ and $(v_k)_{k\in\mathbb N}$ recursively by
\[
w_0:=w,
\qquad
v_0:=v,
\]
and
\[
w_{k+1}:=\frac12(w_k+v_k),
\qquad
v_{k+1}:=w_k!v_k.
\]
Then each $w_k$ and $v_k$ is a representable positive functional, and
\[
w_k\downarrow (w\# v),
\qquad
v_k\uparrow (w\# v).
\]
In particular,
\[
w!v\leq w\#v\leq w\nabla v.
\]
\end{corollary}

\begin{proof}
The assertion follows directly from Theorem~\ref{iteration}. Indeed, for every $a\in\alg$,
\[
(w!v)(a^*a)=\dual{(A!B)a}{a},
\qquad
(w\#v)(a^*a)=\dual{(A\#B)a}{a},
\]
and the same correspondence holds at each step of the iteration.
\end{proof}
 
We now introduce the Busch--Gudder strength function for representable positive functionals. Let $w$ be a representable positive functional on $\alg$, and let $f\in \alg^*$ be a nonzero linear functional. Although the definition of the strength function for positive operators on anti-dual pairs was formulated in terms of conjugate-linear functionals, the distinction is immaterial here: the correspondence $f\mapsto \overline f$ is bijective, and the defining inequality only involves the quantity $|f(x)|^2$. For this reason, and in order to keep the notation closer to the usual convention for functionals on $*$-algebras, we define the strength function along linear functionals.

Thus, for $0\neq f\in\alg^*$, we set
\begin{equation}\label{E:strength-reprfunctional}
\lambda_w(f)
\coloneqq
\sup\set{\lambda\geq0}
{\lambda |f(x)|^2\leq w(x^*x)\qquad (\forall x\in\alg)}.
\end{equation}
In what follows, we investigate the basic properties of the strength function $\lambda_w$ and its behaviour with respect to the operator means introduced above.
If $\alg$ is a $C^*$-algebra, then every positive functional on $\alg$ is automatically continuous. Hence, in this case, it is more natural to work with the topological dual $\alg'$ instead of the algebraic dual $\alg^*$, and to consider the dual pair $(\alg,\alg')$. Accordingly, the strength function $\lambda_w$ may be regarded as a function on $\alg'$. As the following proposition shows, this entails no loss of information: outside the topological dual, the strength function automatically vanishes.

\begin{proposition}\label{P:strength-Cstar}
Let $\alg$ be a unital $C^*$-algebra, let $w$ be a positive functional on $\alg$, let $A:\alg\to\bar{\alg}^*$ be the associated positive operator, and let $0\neq f\in\alg^*$ be a linear functional. Then the following assertions hold:
\begin{enuma}
    \item For every $0\neq f\in\alg^*$, 
    \[ \lambda_w(f)=\lambda_A(\bar f), \] 
    where $\bar f\in\bar{\alg}^*$ is defined by $\bar f(x):=\overline{f(x)}$.
    \item If $v$ is another representable positive functional on $\alg$, then $ w\leq v$ if and only if $\lambda_w\leq \lambda_v.$ 
    \item If $\lambda_w(f)>0$, then $f\in\alg'$ and
    \[
        \|f\|
        \leq
        \frac{\|w\|^{1/2}}{\lambda_w(f)^{1/2}}.
    \]

    \item If $f\in\alg'$, then
    \[
        \lambda_w(f)
        \leq
        \frac{\|w\|}{\|f\|^2}.
    \]
    \item Let $(\hil_w,\pi_w,\zeta_w)$ be the GNS triple associated with $w$. Then $\lambda_w(f)>0$
    if and only if there exists a vector $h_f\in\hil_w$ such that
    \[
        f(a)
        =
        \sipc{\pi_w(a)\zeta_w}{h_f},
        \qquad a\in\alg.
    \]
    In this case, the representing vector $h_f$ is unique, and
    \begin{equation}\label{E:lambdawf=hfnorm}
        \lambda_w(f)=\frac{1}{\|h_f\|^2}.
    \end{equation}

\end{enuma}
\end{proposition}

\begin{proof}
The proof of \textup{(a)} and \textup{(b)} is immediate from the definition and from the corresponding properties of the strength functions of positive operators.

For \textup{(c)}, suppose first that $\lambda_w(f)>0$. Then, for every
$0<\lambda<\lambda_w(f)$, we have
\[
    \lambda |f(a)|^2
    \leq
    w(a^*a)
    \leq
    \|w\|\,\|a\|^2,
    \qquad a\in\alg.
\]
Hence
\[
    |f(a)|
    \leq
    \frac{\|w\|^{1/2}}{\lambda^{1/2}}\|a\|,
    \qquad a\in\alg.
\]
Consequently, $f$ is continuous, and letting $\lambda\uparrow \lambda_w(f)$, we obtain that  
\[
    \|f\|
    \leq
    \frac{\|w\|^{1/2}}{\lambda_w(f)^{1/2}}.
\] 
Assertion \textup{(d)} follows immediately from \textup{(c)}.
 
For \textup{(e)}, recall that
\[
    w(a^*a)=\|\pi_w(a)\zeta_w\|^2,
    \qquad a\in\alg.
\]
If $\lambda_w(f)>0$, then for some $m>0$,
\[
    |f(a)|^2
    \leq
    m\,\|\pi_w(a)\zeta_w\|^2,
    \qquad a\in\alg.
\]
Thus the map
\[
    \pi_w(a)\zeta_w\mapsto f(a),
    \qquad a\in\alg,
\]
is a well-defined bounded linear functional on the dense subspace
$\pi_w(\alg)\zeta_w$ of $\hil_w$. By the Riesz representation theorem, there exists a vector $h_f\in\hil_w$ such that
\[
    f(a)
    =
    \sipc{\pi_w(a)\zeta_w}{h_f},
    \qquad a\in\alg.
\]
Since $\pi_w(\alg)\zeta_w$ is dense in $\hil_w$, this vector is unique. Its norm is given by
\begin{align*}
    \|h_f\|^2&=\inf\set{m\geq 0}{\abs{f(a)}^2\leq m\|\pi_w(a)\zeta_w\|^2~(\forall a\in\alg)}\\
    &=\inf\set{m\geq 0}{\abs{f(a)}^2\leq m\cdot w(a^*a)~(\forall a\in\alg)},
\end{align*}
which proves formula \eqref{E:lambdawf=hfnorm}.
\end{proof}
We now translate the general results on strength functions of operator means into the language of representable positive functionals. The proof follows immediately from the corresponding statements for positive operators on anti-dual pairs. 
\begin{proposition}\label{T:strength-representable-functionals}
Let $\alg$ be a unital $*$-algebra, and let $w$ and $v$ be representable positive functionals on $\alg$. Then:
\begin{enuma}
    \item  $\lambda_{w:v}=\lambda_w:\lambda_v$,
        and $\lambda_{w!v}=\lambda_w!\lambda_v$.
    \item  $  \lambda_{w\nabla v}\geq \lambda_w\nabla\lambda_v,$
     and equality holds if and only if $w$ and $v$ are linearly dependent.
    \item $\lambda_{w\# v}\geq \lambda_w\#\lambda_v.$
\end{enuma}
\end{proposition}
We record the equality case in the geometric-mean inequality separately, in a form adapted to representable positive functionals.

\begin{theorem}\label{T:geom-equality-representable-functionals}
Let $\alg$ be a unital $*$-algebra, let $w$ and $v$ be representable positive functionals on $\alg$, and put $\rho:=w+v$. Let $(\hil_\rho,\pi_\rho,\zeta_\rho)$ denote the GNS triple associated with $\rho$. Then the following statements are equivalent:
\begin{enumi}
    \item $\lambda_{w\# v}=\lambda_w\#\lambda_v,$
    \item  there exist pairwise orthogonal closed reducing subspaces
         $\mathcal M,\mathcal N,\mathcal R\subseteq \hil_\rho$
    for the representation $\pi_\rho$ such that
    \[
    \hil_\rho=\mathcal M\oplus\mathcal N\oplus\mathcal R,
    \]
    and a scalar $0\leq\gamma\leq1$ such that, for every $a\in\alg$,
    \begin{align*}
        w(a)
        &=
        \gamma\,\sipc{\pi_\rho(a)P_{\mathcal R}\zeta_\rho}
        {P_{\mathcal R}\zeta_\rho}
        +
        \sipc{\pi_\rho(a)P_{\mathcal M}\zeta_\rho}
        {P_{\mathcal M}\zeta_\rho},\\
        v(a)
        &=
        (1-\gamma)\,\sipc{\pi_\rho(a)P_{\mathcal R}\zeta_\rho}
        {P_{\mathcal R}\zeta_\rho}
        +
        \sipc{\pi_\rho(a)P_{\mathcal N}\zeta_\rho}
        {P_{\mathcal N}\zeta_\rho}.
    \end{align*}
\end{enumi}
\end{theorem}

\begin{proof}
Let $A,B:\alg\to\bar{\alg}^*$ be the positive operators associated with $w$ and $v$, respectively, and let $C:=A+B$. Then the Hilbert space $\hilc$ associated with $C$ is naturally unitarily equivalent to the GNS Hilbert space $\hil_\rho$ via the correspondence
\[
J_C^*x\mapsto \pi_\rho(x)\zeta_\rho,
\qquad x\in\alg.
\]
Under this identification, the contractions $\widetilde A$ and $\widetilde B$ associated with $A$ and $B$ satisfy
\[
\widetilde B=I_C-\widetilde A,
\]
and, as shown above,
\[
\pi_\rho(a)\widetilde A=\widetilde A\pi_\rho(a),
\qquad a\in\alg.
\]
By Theorem~\ref{T:eqgeom}, equality
\[
\lambda_{w\# v}=\lambda_w\#\lambda_v
\]
holds if and only if, under the above identification, there exist pairwise orthogonal closed subspaces
\[
\mathcal M,\mathcal N,\mathcal R\subseteq \hil_\rho
\]
with
\[
\hil_\rho=\mathcal M\oplus\mathcal N\oplus\mathcal R
\]
and a scalar $0\leq\gamma\leq1$ such that
\[
\widetilde A=\gamma P_{\mathcal R}+P_{\mathcal M},
\qquad
\widetilde B=(1-\gamma)P_{\mathcal R}+P_{\mathcal N}.
\]
Since $\widetilde A$ commutes with $\pi_\rho(a)$ for every $a\in\alg$, its spectral projections commute with $\pi_\rho(a)$ as well. Hence $\mathcal M,\mathcal N,\mathcal R$ are reducing subspaces for $\pi_\rho$.

Finally, using
\[
w(a)=\sipc{\pi_\rho(a)\widetilde A\zeta_\rho}{\zeta_\rho},
\qquad
v(a)=\sipc{\pi_\rho(a)\widetilde B\zeta_\rho}{\zeta_\rho},
\]
and the fact that the projections $P_{\mathcal M},P_{\mathcal N},P_{\mathcal R}$ commute with $\pi_\rho(a)$, we obtain
\begin{align*}
    w(a)
    &=
    \gamma\,\sipc{\pi_\rho(a)P_{\mathcal R}\zeta_\rho}
    {P_{\mathcal R}\zeta_\rho}
    +
    \sipc{\pi_\rho(a)P_{\mathcal M}\zeta_\rho}
    {P_{\mathcal M}\zeta_\rho},\\
    v(a)
    &=
    (1-\gamma)\,\sipc{\pi_\rho(a)P_{\mathcal R}\zeta_\rho}
    {P_{\mathcal R}\zeta_\rho}
    +
    \sipc{\pi_\rho(a)P_{\mathcal N}\zeta_\rho}
    {P_{\mathcal N}\zeta_\rho}.
\end{align*}
Conversely, if such a decomposition exists, then the corresponding operators associated with $w$ and $v$ have precisely the structural form described in Theorem~\ref{T:eqgeom}. Therefore
\[
\lambda_{w\# v}=\lambda_w\#\lambda_v.
\]
This completes the proof.
\end{proof}

\end{document}